\documentclass[11pt]{amsart}

\usepackage{amssymb,latexsym}
\usepackage{amsthm}
\usepackage{amsmath}
\usepackage{amsfonts}
\usepackage{graphicx}
\usepackage[all]{xy}
\usepackage{fancyhdr}
\usepackage[top=1in,bottom=1in,left=1.25in,right=1.25in]{geometry} 

\newtheorem{theorem}{Theorem}[section]
\newtheorem{definition}[theorem]{Defnition}
\newtheorem{remark}[theorem]{Remark}
\newtheorem{proposition}[theorem]{Proposition}
\newtheorem{corollary}[theorem]{Corollary}

\def\Q{\mathbb{Q}}
\def\F{\mathbb{F}}
\def\R{\mathbb{R}}
\def\Z{\mathbb{Z}}
\def\C{\mathbb{C}}
\def\M{\mathcal{M}}
\def\P{\mathcal{P}}
\def\Fm{\mathcal{F}}
\def\lan{\langle}
\def\ran{\rangle}

\def\ra{\rightarrow}
\def\ov{\overline}
\def\wh{\widehat}

\usepackage[pdftex]{hyperref}
\usepackage{fancyhdr}

\begin{document}

\title{On the Hodge-Newton filtration for $p$-divisible groups with additional structures}
\author{Xu SHEN}
\date{}
\address{D\'epartement de Math\'ematiques d'Orsay\\
Universit\'e Paris-Sud 11\\
Orsay, 91405, France} \email{xu.shen@math.u-psud.fr} 

\begin{abstract}
We prove that, for a $p$-divisible group with additional structures over a complete valuation ring of rank one $O_K$ with mixed characteristic $(0,p)$, if the Newton polygon and the Hodge polygon of its special fiber possess a non trivial contact point, which is a break point for the Newton polygon, then it admits a ``Hodge-Newton filtration'' over $O_K$. The proof is based on the theories of Harder-Narasimhan filtration of finite flat group schemes and admissible filtered isocrystals. We then apply this result to the study of some larger class of Rapoport-Zink spaces and Shimura varieties than those in \cite{M2}, and confirm some new cases of Harris's conjecture 5.2 in \cite{Ha}.
\end{abstract}
\maketitle
\tableofcontents

\section{Introduction}

The motivation of this article is to study the cohomology of some unitary group Shimura varieties, namely those introduced in \cite{BW},\cite{VW}. The fixed prime $p$ is assumed to be inert in the quadratic field in the PEL data, so the local reductive groups at $p$ are the quasi-split unitary similitude groups. The generic fibers of these Shimura varieties are the same with those of some special cases studied by Harris-Taylor in \cite{HT}, where they proved the local Langlands correspondence for $GL_n$. The geometry of the special fibers of Harris-Taylor's Shimura varieties is simpler, since in fact one is reduced to the study of one dimensional $p$-divisible groups. For any non-basic Newton polygon strata, the one dimensional $p$-divisible group attached to any point in it admits the splitting  local-\'etale exact sequence. This simple fact plus the theory of Katz-Mazur's ``full set of sections'' lead to the geometric fact that, any non-basic strata in Drinfeld levels is decomposed as some disjoint union of Igusa varieties of first kind defined there. Thus the cohomology of any non-basic strata can be written as a parabolic induction. This reduces the construction of local Langlands correspondence to the study of the basic strata and the corresponding Lubin-Tate spaces. There they got such a conclusion inspired by Boyer's trick in \cite{Bo} for function fields case. Note in Harris-Taylor's case any non-basic Newton polygon has a nontrivial \'etale part, contained in the (generalized) Hodge polygon. For general PEL type Shimura varieties, there are also Newton polygon stratifications. Consider those Shimura varieties which satisfy the condition that, there is some non-basic Newton polygon admitting a nontrivial contact point with the Hodge polygon, and assume this contact point is a break point of the Newton polygon. Under this condition we will wonder that, whether the cohomology of this non-basic strata contains no supercuspidal representations of the associated local reductive group, or even whether the cohomology is some parabolic induction. The methods of Harris-Taylor will hardly work, since in general one knows very little of the geometry of their special fibers in Drinfeld levels. In our cases above, one can draw the pictures of all Newton polygons as in \cite{BW}, 3.1, and find that any non-basic polygon has some nontrivial break contact points with the Hodge polygon.
\\

In this paper we will give a positive answer of the above consideration. The idea is proving the existence of a canonical filtration under the above condition, the so called ``Hodge-Newton filtration'' (see below), for $p$-divisible groups with additional structures, and then passing to their moduli--Rapoport-Zink spaces. This idea is due to Mantovan. In \cite{M2} Mantovan considered this question under the stronger condition that, the Newton polygon coincides with Hodge polygon up to or from on the nontrivial break contact point. Under this stronger condition, Mantovan and Viehmann had proven in \cite{M3} the existence of the Hodge-Newton filtration over characteristic 0 by lifting the corresponding one from characteristic $p$. Note both \cite{M3} and \cite{M2} restricted themselves in just the EL and PEL symplectic cases. In this paper we generalize their results under the more natural condition as above, that is, the Newton polygon admits a nontrivial contact point with the Hodge polygon, and assume this contact point is a break point of the Newton polygon. We will consider also the PEL unitary case. In particular we can prove the desired results for the Shimura varieties in \cite{BW} and \cite{VW} mentioned above.
\\

 The notions of Hodge-Newton decomposition and Hodge-Newton filtration for $F$-crystals was first introduced by Katz in \cite{K}, where under the hypothesis that the Newton polygon and the Hodge polygon of a $F$-crystal possess a non-trivial contact point which is a break point of the Newton polygon, he proved that the $F$-crystal over a perfect field of characteristic $p$ admits a decomposition, such that the two parts of the Newton (resp. Hodge) polygon divided by the point correspond to the Newton (resp. Hodge) polygon of the two sub-$F$-crystals. This can be viewed as a generalization of the multiplicative-bilocal-\'{e}tale filtration for $p$-divisible groups. Katz also proved the existence of Hodge-Newton filtration for $F$-crystals over certain algebras over characteristic $p$.
 \\

In \cite{M3}, Mantovan and Viehmann considered the case of $F$-crystals and $p$-divisible groups with actions of the integer ring of an unramified finite extension of $\Q_p$. They proved that under the stronger condition that the Newton and Hodge polygon coincide up to or from on the contact point, one can lift the Hodge-Newton filtration for $p$-divisible groups from characteristic $p$ to characteristic 0. See theorem 10 of \cite{M3} for the precise statement. In \cite{M2}, Mantovan used this result to prove that, the cohomology of the Rapoport-Zink spaces whose Newton and Hodge polygons satisfy this stronger condition contains no supercuspidal representation of the underlying reductive group defined by the local EL/PEL data.
\\

 In this paper we consider $p$-divisible groups with additional structures over complete valuation rings of rank one of mixed characteristic $(0,p)$, i.e. which are complete extensions of $\Z_p$ for a valuation with values in $\mathbb{R}$. Here additional structures means an action of $O_F$, the integer ring of an unramified finite extension $F|\Q_p$, and a polarization compatible with this action. For precise definition see definition 3.1 in section 3. A typical such $p$-divisible group comes from a $K$-valued point of the simple unramified EL/PEL Rapoport-Zink spaces introduced in \cite{F1}, chapter 2, where $K$ is a complete extension of $\Q_p$ for a rank one valuation. The first main result of this paper is the following. Here to simplify the exposition, we just state the theorem for the PEL cases.

\begin{theorem}
Let $K|\Q_p$ be a complete discrete valuation field with residue field $k$ perfect, $(H,\iota,\lambda)$ be a $p$-divisible group with additional structures over $O_K$.
Assume the (HN) condition: the Newton polygon $Newt(H_k,\iota,\lambda)$ and Hodge polygon $Hdg(H_k,\iota,\lambda)$ possess a contact point $x$ outside their extremal points, which is a break point for the polygon $Newt(H_k,\iota,\lambda)$. Denote by $\hat{x}$ the symmetric point of $x$, and assume $x$ lies before $\hat{x}$, i.e. the horizontal coordinate of $x$ is smaller than that of $\hat{x}$.
Then there are unique subgroups \[(H_1,\iota)\subset(H_2,\iota)\subset(H,\iota)\] of $(H,\iota)$
as $p$-divisible groups with additional structures over $O_K$, such that

\begin{enumerate}
\item $\lambda$ induces isomorphisms \[(H_1,\iota)\simeq
((H/H_2)^D,\iota'),\]\[(H_2,\iota)\simeq ((H/H_1)^D,\iota').\]
Here $(H/H_i)^D$ is the Cartier-Serre dual of $H/H_i$, and $\iota'$ is the induced $O_F$-action on $(H/H_i)^D$, for $i=1,\, 2$;
 \item the induced filtration of the $p$-divisible group with additional structures $(H_k,\iota)$ over $k$ \[(H_{1k},\iota)\subset(H_{2k},\iota)\subset(H_k,\iota)\] is split;
\item
the Newton(resp. Harder-Narasimhan, Hodge) polygons of $(H_{1},\iota),(H_2/H_1,\iota)$, and $(H/H_2,\iota)$ are the parts of the Newton(resp. Harder-Narasimhan, Hodge) polygon of $(H,\iota,\lambda)$ up to $x$, between $x$ and $\hat{x}$, and from $\hat{x}$ on respectively.
\end{enumerate}
When $x=\hat{x}$, then $H_1$ and $H_2$ coincide.

\end{theorem}
In fact, the above theorem holds for more general complete valuation rings of rank one (not necessary discrete) which are extensions of $\Z_p$, with some technical restriction to the so called ``modular'' $p$-divisible groups, see definition 5.6 and theorem 5.7 in section 5.
\\

The proof of this theorem is quite different from that of Mantovan-Viehmann in \cite{M3}. Our ideas are that, firstly using explanation in terms of filtered isocrystals we get a Hodge-Newton filtration for $p$-divisible groups up to isogeny; then using the theory of Harder-Narasimhan filtration for finite flat group schemes over $O_K$ developed in \cite{F2},\cite{F3}, we prove the existence and uniqueness of such a filtration for $p$-divisible groups. More precisely, we define Harder-Narasimhan polygons $HN(H,\iota,\lambda)$ for the $p$-divisible groups and finite flat group schemes with (PEL) additional structures studied here, by adapting the case without additional structures studied in \cite{F3}. Then the crucial points are the following inequalities:\[HN(H,\iota,\lambda)\leq Newt(H_k,\iota,\lambda)\leq Hdg(H_{k},\iota,\lambda) ,\]
and
\[\begin{split}
HN(H,\iota,\lambda)&\leq \frac{1}{m}HN(H[p^m],\iota,\lambda)(m\cdot)\leq HN(H[p],\iota,\lambda)\\
&\leq Hdg(H[p],\iota,\lambda)=Hdg(H_{k},\iota,\lambda).
\end{split} \]
Here  $Newt(H_k,\iota,\lambda)$ and $Hdg(H_{k},\iota,\lambda)$ are the Newton and Hodge polygons of the $p$-divisible group with additional structures $(H_k,\iota,\lambda)$ over $k$, and the inequality $Newt(H_k,\iota,\lambda)\leq Hdg(H_{k},\iota,\lambda)$ is the generalized Mazur's inequality, see \cite{RR}; $Hdg(H[p],\iota,\lambda)$ is the Hodge polygon for the finite flat group scheme $(H[p],\iota,\lambda)$ over $O_K$, defined in section 3 by adapting the Hodge polygon for the case without additional structures defined in \cite{F2}, 8.2 to our situations. By using the explanation of the polygons for $p$-divisible groups in terms of the associated filtered isocrystals, from the above first line of inequalities we deduce that the Harder-Narasimhan polygon also passes the point $\hat{x}$, thus it is necessarily a break point for this polygon. Then by the second line of inequalities one can find a subgroup in the Harder-Narasimhan filtration of $H[p^n]$ for every $n$ large enough. We will show these finite flat group schemes are compatible. Thus they define a $p$-divisible group $H_2$, with its filtered isocrystal as the sub filtered isocrystal corresponding the point $\hat{x}$ of that of $H$.  Similarly there is a $p$-divisible group $H_1$ corresponding to the point $x$. One can then check that the statements in the theorem hold.
\\

This theorem generalizes theorem 10 in \cite{M3}. With this generalization we can study the cohomology of some non-basic Rapoport-Zink spaces $(\M_K)$ exactly as Mantovan did in \cite{M2}, but here we can deal with a larger class of Rapoport-Zink spaces due to our weaker condition. Mantovan's method is to introduce two other towers of moduli spaces $(\P_K),(\Fm_K)$ of $p$-divisible groups with additional structures, corresponding to the levi subgroup $M$ and parabolic subgroup $P$ respectively associated to the nontrivial break contact point $x$. Here $(\P_K)$ are some moduli spaces of $p$-divisible groups, which can be in fact viewed as a type of Rapoport-Zink spaces for the levi subgroup $M$, and $(\Fm_K)$ are deformation spaces of filtered $p$-divisible group by quasi-isogenies analogous to Rapoport-Zink's definition. Fix a prime $l\neq p$. By studying some geometric aspects between the towers $(\P_K)$ and $(\Fm_K)$,  one finds \[H(\P_\infty)_\rho=H(\Fm_\infty)_\rho,\] where $\rho$ is an admissible smooth $\overline{\Q}_l$-representation of $J_b(\Q_p)$, \[H(\P_\infty)_\rho=\sum_{i,j\geq0}(-1)^{i+j}\varinjlim_{K}\textrm{Ext}^j_{J_b(\Q_p)}(H^i_c(\P_K\times\C_p,\overline{\Q}_l(D_\P)),\rho)\]
($D_\P$ is the dimension of $\P_K$), and similarly one has $H(\Fm_\infty)_\rho$, $H(\M_\infty)_\rho$. Here the reason that we consider the formula of cohomology in this type is to apply Mantovan's formula in \cite{M1}. On the other hand, under our condition $(HN)$, thanks to the existence of Hodge-Newton filtration one has \[\M_K=\coprod_{K\setminus G(\Q_p)/P(\Q_p)}\Fm_{K\cap P(\Q_p)}.\]This decomposition has the following application to monodromy representations.

In \cite{Chen}, Chen has constructed some determinant morphisms for the towers of simple unramified Rapoport-Zink spaces. Under the condition that there is no non-trivial contact point of the Newton and Hodge polygons, and assume the conjecture
that \[\pi_0(\widehat{\M})\simeq Im\varkappa\] for the morphism $\varkappa:\widehat{\M}\ra\triangle$ constructed in \cite{RZ} 3.52, Chen proved that the associated monodromy representation under this condition is maximal, and thus the geometric fibers of her determinant morphisms are exactly the geometric connected components, see th\'eor\`eme 5.1.2.1.,and 5.1.3.1. of loc. cit..

Under our condition $(HN)$, the existence of Hodge-Newton filtration implies also that, the monodromy representations associated to the local systems defined by Tate modules of $p$-divisible groups, factor through the parabolic subgroup.
\begin{corollary}Under the above notations, let $\ov{x}$ be a geometric point of the Rapoport-Zink space $\M$, and $\ov{y}$ be its image under the $p$-adic period morphism $\pi:\M\ra\Fm^a$. Then the monodromy representations
\[\rho_{\ov{x}}:  \pi_1(\M,\ov{x})\longrightarrow G(\Z_p)\]and \[\rho_{\ov{y}}:\pi_1(\Fm^a,\ov{y})\longrightarrow G(\Q_p)\] factor through $P(\Z_p)$ and $P(\Q_p)$ respectively.
\end{corollary}
This confirms that the condition ``there is no non-trivial contact point of the Newton and Hodge polygons'' in the chapter 5 of \cite{Chen} is necessary, see the remark in 5.1.5 of loc. cit..
\\

We have the following theorem considering the cohomological application.

\begin{theorem}
Assume the Newton polygon and the Hodge polygon associated to the simple unramified EL/PEL Rapoport-Zink space $\M$ possess a contact point $x$ outside their extremal points which is a break point for the Newton polygon.
Then we have equality of virtual representations of $G(\Q_p)\times W_E$:
\[H(\M_\infty)_\rho=\textrm{Ind}_{P(\Q_p)}^{G(\Q_p)}H(\P_\infty)_\rho .\]
In particular, there is no supercuspidal representations of $G(\Q_p)$ appear in the virtual representation
$H(\M_\infty)_\rho$.
\end{theorem}

 The proof is by adapting the corresponding construction and strategy in \cite{M2} to our cases. In particular, by combining with the main formula of the cohomology of Newton stratas of PEL-type Shimura varieties in \cite{M1}, we have the following corollary.

\begin{corollary}
 For the Shimura varieties studied by \cite{BW},\cite{VW}, $H_c(\overline{Sh}_{\infty}^{(b)}\times\overline{\F}_p,R\Psi_\eta(\overline{\Q}_l))$ can be written as some suitable parabolic induction of virtual representation of some parabolic subgroup of $G(\Q_p)$  for any non-basic strata, thus it contains no supercuspidal representations of $G(\Q_p)$.
\end{corollary}

This confirms Harris's conjecture 5.2 in \cite{Ha} in our case, although the parabolic subgroup may be not the same as that defined in loc. cit. for the non-basic strata. Note these Shimura varieties are out of the cases studied in \cite{M2}, corollary 42. On the other hand, as said above, the similar conclusion for the Shimura varieties with the same generic fibers as that studied in \cite{BW},\cite{VW}, but with $p$ splits in the quadratic field, whose non-basic stratas satisfy the stronger condition in \cite{M2}, was obtained previously by Harris-Taylor in \cite{HT}. There the Hodge-Newton filtration is just the local-\'{e}tale filtration of the $p$-divisible groups. The conclusion in the above corollary that any non-basic strata contains no supercuspidal representations, was once obtained by Fargues in \cite{F1} by using more complicated Lefschetz trace formula methods, initially proposed by Harris in \cite{Ha}. Here our result is more precise and the proof is more natural. With a recent preprint \cite{I}, we remark that we can also apply our main result to the study of cohomology of non-proper Shimura varieties over characteristic 0, and to the geometric realization of local Langlands correspondences.
\\

After submitting this article, both a referee and Mantovan pointed out to me that the method in \cite{M3} should work in the general case, at least after checking some technical details. Nevertheless, our method and idea are totally new, and among others, it is a good application of Fargues' theory. They will have further applications to the study of the $p$-adic geometry of Shimura varieties and Rapoport-Zink spaces. For example, in \cite{Sh} we used some results here in a crucial way, namely the Harder-Narasimhan polygon passes also the contact points under the Hodge-Newton condition (corollary 5.3). This implies that the compactness of the fundamental domain constructed there for some unitary group Rapoport-Zink spaces, and the compactness is very important for establish the Lefschetz trace formula, for more details see loc. cit.. Also in a future work, we will study the $p$-adic geometry of some PEL type Shimura varieties by our Harder-Narasimhan filtration for finite flat group schemes and $p$-divisible groups. Eventually these facts will help us to develop a good theory of $p$-adic automorphic forms in that case.
\\

This paper is organized as follows. In section 2, we review briefly Fargues's theory of Harder-Narasimhan filtration of finite flat group schemes. In section 3, we first define the reductive groups which we will work with, then define the Harder-Narasimhan polygons for $p$-divisible groups with additional structures and prove some inequalities (those in corollary 3.7 and 3.11) between various polygons. In section 4, we explain some of these inequalities in terms of filtered isocrystals with additional structures. In section 5, we prove our first main theorem on Hodge-Newton filtration. In section 6, we first review the definition of Rapoport-Zink spaces and Mantovan's construction of two other towers of spaces, then use the previous results we obtained to deduce the second main theorem. We also examine the $p$-adic period morphism and deduce the result on monodromy representation. In the last section we apply the second main theorem to the study of the cohomology of some Shimura varieties.\\
 \\
\textbf{Acknowledgments.} I would like to thank Prof. Laurent Fargues sincerely, since without his guide this paper would not be accomplished. Special thanks go to Prof. Elena Mantovan for her explanation. I should also thank the referee for careful reading and suggestions.

\section{Harder-Narasimhan filtration of finite flat group schemes}

In this section we recall briefly the theory of Harder-Narasimhan filtration of finite flat group schemes which is presented in detail in \cite{F2}, but see also \cite{F3}.

Let $K|\Q_p$ be a complete rank one valuation field extension, $O_K$ be the integer ring of $K$, and $\mathcal{C}$ be the exact category of commutative finite flat group schemes with order some power of $p$ over $SpecO_K$. For $G\in\mathcal{C}$, recall there is an operation of scheme theoretic closure which is the inverse of taking generic fibers, and which induces a bijection between the following two finite sets
\[\{\textrm{closed subgroups of}\,G_K\}\stackrel{\sim}{\longrightarrow}\{\textrm{finite flat subgroups of }\,G\,\textrm{over}\,O_K\}.\]

 There are two additive functions \[ht: \mathcal{C}\ra \mathbb{N}\]\[deg: \mathcal{C}\ra\R_{\geq0},\] where $htG$ is the height of $G\in\mathcal{C}$, and $degG$ is the valuation of the discriminant of $G$ which is defined as
  \[degG=\sum a_i,\, \textrm{if}\; \omega_G=\bigoplus O_K/p^{a_i}O_K.\] We will use the following properties of the function $deg$.
  \begin{proposition}[\cite{F2}, Corollaire 3]\begin{enumerate}
  \item
  Let $f: G\ra G'$ be a morphism of finite flat group schemes over $O_K$ such that it induces an isomorphism on their generic fibers. Then we have \[degG\leq degG'.\]Moreover, $f$ is an isomorphism if and only if $degG=degG'$.
  \item If
  \[0\ra G'\stackrel{u}{\ra} G\stackrel{v}{\ra} G''\] is a sequence of finite flat group schemes, such that $u$ is a closed immersion, $v\circ u=0$, and the induced morphism $G/G'\ra G''$ is an isomorphism on their generic fibers. Then we have
  \[degG\leq degG'+degG'',\]with the equality holds if and only if $v$ is a fppf epimorphism, i.e. flat. In this case we have an exact sequence
  \[  0\ra G'\stackrel{u}{\ra} G\stackrel{v}{\ra} G''\ra 0.\]
  \end{enumerate}
  \end{proposition}
See \cite{F2} section 3 for more properties of the function $deg$.

   For a group scheme $0\neq G\in \mathcal{C}$, we set
\[\mu(G):=\frac{degG}{htG},\]and call it the slope of $G$. The basic properties of the slope function are as follow.
\begin{itemize}
 \item One always has $\mu(G)\in [0,1]$, with $\mu(G)=0$ if and only if $G$ is \'{e}tale and $\mu(G)=1$ if and only if $G$ is multiplicative.
 \item If $G^D$ is the Cartier dual of $G$ then $\mu(G^D)=1-\mu(G)$.
 \item For a $p$-divisible group $H$ of dimension $d$ and height $h$ over $O_K$, then for all $n\geq 1$ one has $\mu(H[p^n])=\frac{d}{h}$.
 \item If \[0\ra G'\ra G\ra G''\ra0\] is an exact sequence of non trivial groups in $\mathcal{C}$, then we have $inf\{\mu(G'),\mu(G'')\}\leq \mu(G)\leq sup\{\mu(G'),\mu(G'')\}$, and if $\mu(G')\neq \mu(G'')$ we have in fact $inf\{\mu(G'),\mu(G'')\}< \mu(G)< sup\{\mu(G'),\mu(G'')\}$.
 \item If $f: G\ra G'$ is a morphism which induces an isomorphism between their generic fibers, then we have $\mu(G)\leq \mu(G')$, with equality holds if and only if $f$ is an isomorphism.
 \item If \[0\longrightarrow G'\stackrel{u}{\longrightarrow} G\stackrel{v}{\longrightarrow} G''\] is a sequence of non trivial groups such that $u$ is a closed immersion, $u\circ v=0$, and the morphism induced by $v$
\[G/G'\ra G''\] is an isomorphism in generic fibers, then we have \begin{itemize}
 \item $\mu(G)\leq sup\{\mu(G'),\mu(G'')\}$;
 \item if $\mu(G')\neq \mu(G'')$ then $\mu(G)< sup\{\mu(G'),\mu(G'')\}$;
 \item if $\mu(G)=sup\{\mu(G'),\mu(G'')\}$ then $\mu(G)=\mu(G')=\mu(G'')$ and the sequence $0\ra G'\ra G\ra G''\ra0$ is exact.
 \end{itemize}

 \end{itemize}

For a group $0\neq G\in \mathcal{C}$, we call $G$ semi-stable if for all $0 \subsetneq G'\subset G$ we have $\mu(G')\leq \mu(G)$. In \cite{F2}, Fargues proved the following theorem.
\begin{theorem}[\cite{F2}, Th\'eor\`eme 2]
There exists a Harder-Narasimhan type filtration for all $0\neq G\in\mathcal{C}$, that is a chain of finite flat subgroups in $\mathcal{C}$
\[0=G_0\subsetneq G_1\subsetneq \cdots \subsetneq G_r=G,\]with the group schemes $G_{i+1}/G_i$ are semi-stable for all $i=0,\dots,r-1$, and \[\mu(G_1/G_0)>\mu(G_2/G_1)>\cdots >\mu(G_r/G_{r-1}).\] Such a filtration is then uniquely characterized by these properties.
 \end{theorem}

 So $G$ is semi-stable if and only if its Harder-Narasimhan filtration is $0\subsetneq G$. We can define a concave polygon $HN(G)$ of any $0\neq G\in\mathcal{C}$ by its Harder-Narasimhan filtration, and call it the Harder-Narasimhan polygon of $G$. It is defined as function
\[HN(G): [0,htG]\ra [0,degG],\]
such that \[HN(G)(x)=degG_i+\mu(G_{i+1}/G_i)(x-htG_i)\]if $x\in [htG_i,htG_{i+1}]$. We will also identify $HN(G)$ with its graph, that is a polygon with starting point $(0,0)$, terminal point $(htG,degG)$, and over each interval $[htG_i,htG_{i+1}]$ it is a line of slope $\mu(G_{i+1}/G_i)$. An important property of this polygon is that (proposition 7 in loc. cit.), for all finite flat subgroups $G'\subset G$, the point $(htG',degG')$ is on or below the polygon $HN(G)$, that is $HN(G)$ is the concave envelop of the points $(htG',degG')$ for all $G'\subset G$. We denote by $\mu_{max}(G)$ the maximal slope of $HN(G)$, and $\mu_{min}(G)$ the minimal slope of $HN(G)$.

We recall some useful facts.
\begin{proposition}[\cite{F2}, Proposition 8]Let $G_1$ and $G_2$ be two finite flat group schemes over $O_K$. Suppose that $\mu_{min}(G_1)>\mu_{max}(G_2)$. Then we have \[Hom(G_1,G_2)=0.\]
\end{proposition}
\begin{proposition}[\cite{F2}, Proposition 10]Let $0\ra G'\ra G\stackrel{v}{\ra} G''\ra0$ be an exact sequence of finite flat group schemes in $\mathcal{C}$. Suppose $\mu_{min}(G')\geq\mu_{max}(G'')$. If $0=G'_0\subsetneq G'_1\subsetneq\cdots\subsetneq G'_{r'}=G'$ is the Harder-Narasimhan filtration of $G'$ and
$0=G''_0\subsetneq G''_1\subsetneq\cdots\subsetneq G''_{r''}=G''$ that of $G''$, then the Harder-Narasimhan filtration of $G$ is
\[0=G'_0\subsetneq G'_1\subsetneq\cdots\subsetneq G'_{r'}=G'\subsetneq v^{-1}(G''_0)\subsetneq v^{-1}(G_1'')\subsetneq\cdots\subsetneq v^{-1}(G''_{r''})=G\] if $\mu_{min}(G')>\mu_{max}(G'')$, and
\[0=G'_0\subsetneq G'_1\subsetneq\cdots\subsetneq G'_{r'-1}\subsetneq v^{-1}(G_1'')\subsetneq\cdots\subsetneq v^{-1}(G''_{r''})=G\]if $\mu_{min}(G')=\mu_{max}(G'')$.
 In particular the extension of two semi-stable groups of the same slope $\mu$ is semi-stable of slope $\mu$.
 \end{proposition}

The Harder-Narasimhan filtration of finite flat group schemes is compatible with additional structures. Firstly, the Harder-Narasimhan filtration of $0\neq G\in \mathcal{C}$ is stable under $End(G)$. So if $\iota: R\ra End(G)$ is some action of an $O_K$-algebra $R$, then every subgroup $G_i$ in the Harder-Narasimhan filtration of $G$ is a $R$-subgroup via $\iota$. Secondly, if the Harder-Narasimhan filtration of $G$ is
\[0=G_0\subsetneq G_1\subsetneq \cdots \subsetneq G_r=G\]with slopes $\mu_1>\cdots>\mu_r$,
then the Harder-Narasimhan filtration of the Cartier dual $G^D$ of $G$ is
\[0=(G/G_r)^D\subsetneq (G/G_{r-1})^D\subsetneq\cdots\subsetneq (G/G_1)^D\subsetneq G^D\]
with slopes $1-\mu_r>\cdots>1-\mu_1$. In particular, if $\lambda: G\stackrel{\sim}{\ra}G^D$ is a polarization, then it induces isomorphisms \[ G_i\simeq (G/G_{r-i})^D, i=1,\dots,r\]and thus $\mu_i+\mu_{r+1-i}=1,i=1,\dots,r$.

\section{Polygons and inequalities}

We start by defining the reductive groups which we will work with. They are defined by the simple unramified EL/PEL data for defining some special Rapoport-Zink spaces as the chapter 2 of \cite{F1}.

More precisely, let $F| \Q_p$ be a finite unramified extension of degree $d$, $V$ be a finite dimensional $F$-vector space. In the EL case, let $G=\textrm{Res}_{F|\Q_p}GL(V)$, the Weil scalar restriction of the automorphism group of $V$ as a $F$-vector space. In the PEL symplectic case, we assume further there is a hermitien symplectic pairing $\lan,\ran: V\times V\ra \Q_p$, which is such that there exists an autodual lattice $\Lambda$ for $\lan,\ran$ in $V$. In the PEL unitary case, besides the above $\lan,\ran$ and $\Lambda$, we assume there is a non trivial involution $\ast$ on $F$, compatible with $\lan,\ran$, which means that $\lan bu,v\ran=\lan u,b^\ast v\ran$ for all $b \in F, u,v\in V$. We define a reductive group $G$ over $\Q_p$ for these PEL cases, such that for all $\Q_p$-algebra $R$,
\[G(R)=\{g\in \textrm{End}_{F\otimes R}(V_R)|gg^\#\in R^\times\},\]
here $\#$ is the involution on $\textrm{End}_F(V)$ induced by $\lan,\ran$. Then we have $G\subset \textrm{Res}_{F|\Q_p}GSp(V,\lan,\ran')$ (symplectic case) or $G\subset \textrm{Res}_{F_0|\Q_p}GU(V,\lan,\ran')$ (unitary case), where $F_0=F^{\ast=1}$ and $\lan,\ran': V\times V\ra F^{\ast=1}$ is a suitable pairing coming from $\lan,\ran$ in the PEL cases ($\ast$ may be trivial). The rational Tate module of a $p$-divisible group with additional structures (see definition 3.1 in the following) will naturally give arise such an EL data $(F,V)$ or a PEL data $(F,\ast,V,\lan,\ran)$ ($\ast$ may be trivial). Note these reductive groups $G$ are unramified over $\Q_p$, and for the PEL cases there is a similitude morphism $c: G\ra \mathbb{G}_m, g\mapsto gg^\#$.

Let $\mathbb{D}$ be the pro-algebraic torus with character group $\Q$. We will be interested in the set
\[N(G)=(\textrm{Int}G(L)\setminus Hom_L(\mathbb{D},G))^{\lan\sigma\ran},\]
where $L=\textrm{Frac}W(\overline{\F}_p), \sigma$ is the Frobenious of $L$ over $\Q_p$. This set generalizes the classical notion of Newton polygon associated to an $F$-isocrystal, see the introduction of \cite{Ch} or \cite{RR}, section 1. Since the group $G$ is unramified, we can choose a maximal torus $T$ contained in a Borel subgroup $B$ of $G$ defined over $\Q_p$. Let $A\subset T$ be the maximal splitting torus contained in $T$, $W$ (resp. $W_0$) be the absolute (resp. relative) Weyl group, then we have
\[\begin{split}
N(G)&=(X_\ast(T)_\Q/W)^{\textrm{Gal}(\overline{\Q}_p/\Q_p)}\\
&=(\overline{C}\cap X_\ast(T)_\Q)^{\textrm{Gal}(\overline{\Q}_p/\Q_p)}\\
&=X_\ast(A)_\Q/W_0\\
&=:\overline{C}_\Q,
\end{split}\]
where $\overline{C}\subset X_\ast(T)_\R$ is the Weyl chamber associated to $B$. Recall there is an order in $N(G)$, cf. \cite{RR} section 2, such that for all $x,x'\in \overline{C}_\Q$,
\[x\leq x' \Leftrightarrow x'-x=\sum_{\alpha\in \Delta_B}n_\alpha \alpha^{\vee}, n_\alpha \in \Q_{\geq0}.\] Here $\Delta_B$ denotes the set of simple roots determined by $B$, $\alpha^{\vee}$ denotes the co-root corresponding to $\alpha$. Note $N(\cdot)$ is in fact an ordered set-valued functor on the category of connected reductive algebraic groups.

We want to make the elements in the above cone $N(G)$ ``visible'', i.e. as polygons defined over some suitable interval. Let $n=\textrm{dim}_FV$, then in the EL case after choosing a base of the $F$-vector space $V$, we have $\textrm{Res}_{F|\Q_p}GL(V)=\textrm{Res}_{F|\Q_p}GL_n$, and for this case we can explicitly calculate
\[N(G)=\{(x_i)\in \Q^n|x_1\geq x_2\geq \cdots \geq x_n\}:=\Q^n_+ ,\]
which we will identify with the set of concave polygons with rational slopes over the interval $[0,n]$, see \cite{F1} 2.1. For PEL symplectic case,
\[G\subset \textrm{Res}_{F|\Q_p}GSp(V,\lan,\ran')\subset \textrm{Res}_{F|\Q_p}GL_n,\]
since $N(\cdot)$ is a functor on the category of connected reductive algebraic groups, we have an order preserving map
\[N(G)\longrightarrow N(\textrm{Res}_{F|\Q_p}GL_n),\] which is injective, and the image corresponds to symmetric polygons, cf. \cite{W}. For PEL unitary case,
\[G\subset \textrm{Res}_{F_0|\Q_p}GU(V,\lan,\ran')\subset \textrm{Res}_{F_0|\Q_p}\textrm{Res}_{F|F_0}GL(V)=\textrm{Res}_{F|\Q_p}GL_n,\]
similarly we have an order preserving map
\[N(G)\longrightarrow N(\textrm{Res}_{F|\Q_p}GL_n),\] which is also injective, and the image also corresponds to symmetric polygons, cf. loc. cit..

After the preliminary on the reductive groups $G$, we define $p$-divisible groups with additional structures. Since $p$-divisible groups are closely related to finite flat group schemes, we will also consider the related notions for these group schemes.

\begin{definition}
Let $S$ be a formal scheme and $F|\Q_p$ be a finite unramified extension. By a $p$-divisible group with additional structures over $S$, we mean
\begin{itemize}
\item in the EL case, a pair $(H,\iota)$, where $H$ is a $p$-divisible group over $S$, and $\iota: O_F\ra \textrm{End}(H)$ is a homomorphism of algebras;
\item in the PEL symplectic case, a triplet $(H,\iota,\lambda)$, where $H$ is a $p$-divisible group over $S$, $\iota: O_F\ra \textrm{End}(H)$ is homomorphism of algebras, $\lambda: (H,\iota)\stackrel{~}{\ra}(H^D,\iota^D)$ is a polarization, i.e. an $O_F$-equivariant isomorphism of $p$-divisible groups. Here $H^D$ is the Cartier-Serre dual of the $p$-divisible group $H$, $\iota^D: O_F\ra \textrm{End}(H^D)=\textrm{End}(H)^{opp}$ is induced by $\iota$, such that $\lambda^D=-\lambda$, under the identification $H=H^{DD}$;
\item in the PEL unitary case, a triplet $(H,\iota,\lambda)$, where $H,\iota$ is similar as the symplectic case, $\lambda: (H,\iota)\stackrel{~}{\ra}(H^D,\iota^D\circ \ast)$ is a polarization, $\ast$ is a nontrivial involution on $F$. Here $\iota^D$ is as above, but such that $\lambda^D=\lambda$, under the identification $H=H^{DD}$.
\end{itemize}
Similarly, on can define finite locally free (=flat, in the case $S$ is noetherian or the spec of a local ring) group schemes with additional structures in the same way.
\end{definition}
 If $(H,\iota,\lambda)$ is a $p$-divisible group with additional structures in the PEL cases, then for every integer $n\geq 1, (H[p^n],\iota,\lambda)$ is a finite locally free group scheme with the naturally induced additional structures. Similar remark holds for the $EL$ case.

In the rest of this section, let $K|\Q_p$ be a complete field extension for a rank one valuation, $O_K$ be the ring of integers of $K$, $k$ be the residue field, and $F|\Q_p$ be a finite unramified extension of degree $d$. For the PEL (unitary) case we also assume there is an involution $\ast$ on $F$. We shall mostly be interested only in $p$-divisible groups and finite flat group schemes with additional structures for $F|\Q_p$ over $O_K$, $k$, and $\overline{k}$, a fixed algebraic closure of $k$. Let $\underline{H}/O_K$ denote a $p$-divisible group with additional structures over $O_K$ for the EL case ($\underline{H}=(H,\iota)$) or PEL cases ($\underline{H}=(H,\iota,\lambda)$), then $\underline{H}_k$ (resp. $\underline{H}_{\overline{k}}$) is a $p$-divisible group with additional structures over $k$ (resp. $\overline{k}$). Kottwitz defined the Newton polygon $Newt(\underline{H}_{\overline{k}})$ and the Hodge polygon $Hdg(\underline{H}_{\overline{k}})$ as elements in $N(G)$, see \cite{Ko1},\cite{Ko2}. Here the reductive group $G$ is defined by the rational Tate module of $H$ with the induced additional structures as in the beginning of this section.  Assume $htH=dn$, then via the injection
\[N(G)\hookrightarrow N(\textrm{Res}_{F|\Q_p}GL_n) ,\] one can explain them as polygons as following:
\[\begin{split} Newt(\underline{H}_{\overline{k}}): \,[0,n]&\longrightarrow [0,dimH/d]\\
&x\mapsto \frac{1}{d}Newt(H_{\overline{k}})(dx).
\end{split}\]
Here $Newt(H_{\overline{k}})$ is the concave Newton polygon of $H_{\overline{k}}$ defined by the Dieudonn\'{e}-Manin decomposition of its isocrystal. The Hodge polygon is
\[Hdg(\underline{H}_{\overline{k}})=\frac{1}{d}\sum_{i\in \Z/d\Z}Hdg_i(H_{\overline{k}}),\] where \[Hdg_i(H_{\overline{k}}): [0,n]\ra[0,dimH/d]\] is the polygon defined by the relative position of $(M_i,V(M_{i+1}))$ in $M_{i\Q}$. Here $M$ is the covariant Dieudonn\'{e} module, $V$ is the Verschiebung. Under the action of $O_F$, we have \[M=\bigoplus_{i\in \Z/d\Z}M_i, M_i=\{m\in M|a\cdot m=\sigma^i(a)m, \forall a\in O_F\}, \forall \,i \in \Z/d\Z.\] Note $Gal(F/\Q_p)=\{\sigma^i|i\in \Z/d\Z\}$. One can check that these two polygons don't depend on the choice of the algebraic closure. Thus we can define the Newton (resp. Hodge) polygon of $\underline{H}_k$ by $Newt(\underline{H}_k):=Newt(\underline{H}_{\overline{k}})$ (resp. $Hdg(\underline{H}_k):=Hdg(\underline{H}_{\overline{k}})$).

Let $\underline{H}/O_K$ be a $p$-divisible group with additional structures as above. We are going to define the Harder-Narasimhan polygon of $\underline{H}/O_K$, and compare this polygon with the above polygons. We first consider the case of finite flat group schemes with additional structures with order some power of $p$. We now use $\underline{H}$ to denote such a finite flat group scheme with additional structures. Recall the underlying finite flat group scheme $H/O_K$ admits a unique Harder-Narasimhan filtration. Let \[HN(H): [0,htH]\ra [0,degH]\]be the concave polygon associated to this filtration.

\begin{definition}
Consider \[\begin{split}
HN(H,\iota): \,[0,htH/d]&\longrightarrow [0,degH/d]\\
&x\mapsto \frac{1}{d}HN(H)(dx)
\end{split}\] in the EL case, and $HN(H,\iota,\lambda):=HN(H,\iota)$ in the PEL cases, which is symmetric. $HN(H,\iota)$ and $HN(H,\iota,\lambda)$ are called the Harder-Narasimhan polygons of the finite flat group schemes $\underline{H}/O_K$ with additional structures.
\end{definition}

Now we define the Harder-Narasimhan polygon for $p$-divisible groups with additional structures. To be more concrete on notations, assume we are in the PEL cases, although all the following work for the EL case, which is simpler. Let $(H,\iota,\lambda)/O_K$ be a $p$-divisible group with additional structures, then we get a family of finite flat group schemes with additional structures $(H[p^m],\iota,\lambda)/O_K$.

\begin{proposition}
The sequence of functions
\[\begin{split}[0,htH/d]&\ra[0,dimH/d]\\
&x\mapsto \frac{1}{m}HN(H[p^m],\iota,\lambda)(mx)
\end{split}\]
uniformly converge when $m\ra \infty $ to a concave continuous ascending function
\[HN(H,\iota,\lambda): [0,htH/d]\ra[0,dimH/d] ,\] which is equal to \[\inf_{m\geq 1}\frac{1}{m}HN(H[p^m],\iota,\lambda)(mx)\] and moreover, $HN(H,\iota,\lambda)(0)=0, HN(H,\iota,\lambda)(htH/d)=dimH/d$.
\end{proposition}

\begin{proof}
Essentially the same with the proof of th\'eor\`eme 2 in \cite{F3}, or one can easily deduce this from the results there.
\end{proof}

\begin{definition}
We call the function $HN(H,\iota,\lambda)$ or its graph which we denote by the same symbol,
the Harder-Narasimhan polygon of the $p$-divisible group with PEL additional structures $(H,\iota,\lambda)$ over $O_K$.
Similarly we can define the Harder-Narasimhan polygon of $p$-divisible groups with EL structures.
\end{definition}
In fact, we have the following inequality which can be also easily deduced from the corresponding result in \cite{F3}: for all $i\geq 1, x\in [0,i\frac{htH}{d}]$, and all $m\geq 1$,
\[\frac{1}{m}HN(H[p^{im}],\iota,\lambda)(mx)\leq HN(H[p^i],\iota,\lambda)(x).\]

For the $p$-divisible group with additional structures $(H_{\overline{k}},\iota,\lambda)$ over $\overline{k}$, we have the Newton and Hodge polygons $Newt(H_{\overline{k}},\iota,\lambda)$, $Hdg(H_{\overline{k}},\iota,\lambda)$ respectively defined by Kottwitz. By Rapoport-Richartz, we have the generalized Mazur's inequality (see \cite{RR}):
\[Newt(H_{\overline{k}},\iota,\lambda)\leq Hdg(H_{\overline{k}},\iota,\lambda) .\]
We assume $H/O_K$ is a ``modular'' $p$-divisible group in the sense of definition 25 in \cite{F3}, see also definition 5.6 in section 5. The following theorem is one of the main theorems of \cite{F3}.
\begin{theorem}[\cite{F3}, Th\'eor\`eme 21]Let $K$ be as above and $H/O_K$ be a $p$-divisible group over $O_K$. When the valuation ring $O_K$ is not discrete we assume that $H/O_K$ is ``modular''. Then we have the following inequality
\[HN(H)\leq Newt(H_k).\]
\end{theorem}

One gets from their definitions the following generalization.
\begin{proposition} Let $(H,\iota,\lambda)$ be a $p$-divisible group with additional structures over $O_K$. When $O_K$ is not discrete we assume that $H/O_K$ is ``modular''. Then we have the following inequality
\[HN(H,\iota,\lambda)\leq Newt(H_k,\iota,\lambda).\]
\end{proposition}

Combined with the generalized Mazur's inequality we get the following corollary.
\begin{corollary}
Under the above assumption, we have inequalities
\[HN(H,\iota,\lambda)\leq Newt(H_k,\iota,\lambda)\leq Hdg(H_{k},\iota,\lambda).\]
\end{corollary}

Now we return to the case of finite flat group schemes with additional structures $(H,\iota)$ or $(H,\iota,\lambda)$. Fargues defined the Hodge polygon of $H$ from $\omega_H$, see \cite{F2} 8.2. We would like to generalize his definition to define a Hodge polygon of $\underline{H}$, which contains the information of the additional structures.
First, assume $F$ admits an imbedding in $K$, then we have also the decomposition
\[\omega_H=\bigoplus_{\tau:F\hookrightarrow K}\omega_{H,\tau},\, \omega_{H,\tau}=\{m\in \omega_H| a\cdot m=\tau(a)m, \forall a\in O_F\},\forall \tau:F\hookrightarrow K.\]

\begin{definition} Under the above assumption,
\begin{enumerate}
\item \[Hdg_{\tau}: [0,htH/d]\ra [0,degH/d] \] is the polygon such that
\[Hdg_{\tau}(i)=deg(\omega_{H,\tau})-\nu(Fitt_i\omega_{H,\tau}), 0\leq i\leq htH/d,\] where $\nu$ is the valuation on $K$ such that $\nu(p)=1$ and $Fitt_iM$ means the $i$-th Fitting ideal of an $O_K$-module $M$. In particular, $degM=\nu(Fitt_0M)$.
\item In the EL case, the Hodge polygon of $(H,\iota)$ is
\[Hdg(H,\iota)=\frac{1}{d}\sum_{\tau:F\hookrightarrow K}Hdg_\tau: [0,htH/d]\ra[0,degH/d].\] In the PEL cases, the Hodge polygon of $(H,\iota,\lambda)$ is $Hdg(H,\iota,\lambda):=Hdg(H,\iota)$, which is then symmetric.
\end{enumerate}
\end{definition}
For the general case, choose a complete field $K'\supset K$, such that $F\hookrightarrow K'$, we define $Hdg(\underline{H}):=Hdg(\underline{H}_{O_{K'}})$. One can check this definition doesn't depend on the choice of $K'$.
\begin{remark}
The above definition is compatible with the Hodge polygon defined by Kottwitz, in the sense that if $\underline{H}/O_K$ is a $p$-divisible group with additional structures, then $Hdg(\underline{H}[p])=Hdg(\underline{H}_{k})$.
\end{remark}

\begin{proposition}
Let $\underline{H}/O_K$ be a finite flat group scheme with additional structures of order a power of $p$. Assume there is a $p$-divisible group with additional structures $\underline{\mathcal{G}}/O_K$ such that $\underline{H}$ admits an imbedding $\underline{H}\hookrightarrow \underline{\mathcal{G}}$. Then we have
\[HN(\underline{H})\leq Hdg(\underline{H}).\]
\end{proposition}
\begin{proof}
We may assume we are in the EL case: $\underline{\mathcal{G}}=(\mathcal{G},\iota), \underline{H}=(H,\iota)$. We may also assume $F\hookrightarrow K$. Then $\forall \tau: F\hookrightarrow K$, we have an exact sequence of covariant Dieudonn\'{e} modules:
\[0\ra \mathbb{D}(\mathcal{G}_{\overline{k}})_\tau\ra \mathbb{D}((\mathcal{G}/H)_{\overline{k}})_\tau\ra\mathbb{D}(H_{\overline{k}})_\tau\ra 0.\]
To simplify notation, let $M=\mathbb{D}(\mathcal{G}_{\overline{k}}), M'=\mathbb{D}((\mathcal{G}/H)_{\overline{k}})$.
Then the length of the $W(\overline{k})$-module $\mathbb{D}(H_{\overline{k}})_\tau$ is the index $[M'_\tau: M_\tau]$ of $M_\tau$ in $M'_\tau$. Consider the Lie algebras $Lie(\mathcal{G}_{\overline{k}})$ and $Lie((\mathcal{G}/H)_{\overline{k}})$ of $\mathcal{G}_{\overline{k}}$ and $(\mathcal{G}/H)_{\overline{k}}$ respectively, then we have also the decompositions induced by the $O_F$ action:
\[Lie(\mathcal{G}_{\overline{k}}=\bigoplus_\tau Lie(\mathcal{G}_{\overline{k}})_\tau,\\
Lie((\mathcal{G}/H)_{\overline{k}}=\bigoplus_\tau Lie((\mathcal{G}/H)_{\overline{k}})_\tau.\]
By Diedonn\'{e} theory we get the identities
\[Lie(\mathcal{G}_{\overline{k}})_\tau=M_\tau/VM_{\sigma\tau},\\
Lie((\mathcal{G}/H)_{\overline{k}})_\tau=M'_\tau/VM'_{\sigma\tau}.\]
Here $\sigma$ is the Frobenious for the extension $F|\Q_p$, and $V$ is the Verschiebung on $M$.
Then the isogeny $\mathcal{G}\ra\mathcal{G}/H$ induces \[rank(Lie(\mathcal{G}_{\overline{k}})_\tau)=rank(Lie((\mathcal{G}/H)_{\overline{k}})_\tau),\] which can be rewritten as
\[[M_\tau:VM_{\sigma\tau}]=[M'_\tau:VM'_{\sigma\tau}].\]
Consider the following commutative diagram

\[\xymatrix{
0\ar[r] & M_\tau \ar[r] & M'_\tau\ar[r] &\mathbb{D}(H_{\overline{k}})_\tau\ar[r] &0\\
0\ar[r] & M_{\sigma\tau} \ar[r] \ar[u]_{V}& M'_{\sigma\tau}\ar[r] \ar[u]_{V}&\mathbb{D}(H_{\overline{k}})_{\sigma\tau}\ar[r]\ar[u]_{V} &0 ,}\] then
\[\begin{split}
[M'_\tau:VM_{\sigma\tau}]&=[M'_\tau:M_\tau][M_{\tau}:VM_{\sigma\tau}]\\
&=[M'_\tau:VM'_{\sigma\tau}][VM'_{\sigma\tau}:VM_{\sigma\tau}].
\end{split}\]
Thus \[[M'_\tau:M_\tau]=[VM'_{\sigma\tau}:VM_{\sigma\tau}]=[M'_{\sigma\tau}:M_{\sigma\tau}] ,\]
i.e. the lengths of these $W(\overline{k})$-modules $\mathbb{D}(H_{\overline{k}})_\tau$ for any $\tau$ are the same.
We can then conclude that $\mathbb{D}(H_{\overline{k}})_\tau$ is generated by $ht_{O_F}H=htH/d$ elements. So
\[\omega_{H_{\overline{k},\tau}}=(\mathbb{D}(H_{\overline{k}})_\tau/V\mathbb{D}(H_{\overline{k}})_{\sigma\tau})^\ast=
\omega_{H,\tau}\otimes\overline{k}\] is generated by $htH/d$ elements, where $M^\ast=Hom(M,W(\overline{k}))$ for a $W(\overline{k})$-module $M$. By Nakayama lemma, $\omega_{H,\tau}$ is generated by $htH/d$ elements.

Now for any subgroup $(H',\iota)\subset (H,\iota)$, we have an exact sequence
\[0\ra \omega_{H/H',\tau}\ra \omega_{H,\tau}\ra\omega_{H',\tau}\ra 0.\]
By the basic properties of Fitting ideals,
\[Fitt_{ht_{O_F}H'}\omega_{H',\tau}Fitt_0\omega_{H/H',\tau} \subset Fitt_{ht_{O_F}H'}\omega_{H,\tau}.\]
But \[Fitt_{ht_{O_F}H'}\omega_{H',\tau}=O_K\] by the above paragraph, thus
\[\begin{split}
\nu(Fitt_0\omega_{H/H',\tau})&=deg\omega_{H/H',\tau}\\
&=deg\omega_{H,\tau}-deg\omega_{H',\tau}\\
&\geq \nu(Fitt_{ht_{O_F}H'}\omega_{H,\tau}),
\end{split}\]
and \[Hdg_\tau(ht_{O_F}H')=deg\omega_{H,\tau}-\nu(Fitt_{ht_{O_F}H'}\omega_{H,\tau})\geq deg\omega_{H',\tau}.\]
We sum the above inequality over $\tau: F\hookrightarrow K$, then divide by $d$ to get
\[\frac{1}{d}\sum_{\tau:F\hookrightarrow K}Hdg_\tau(ht_{O_F}H')\geq \frac{1}{d}deg\omega_{H'}.\]
But the left hand side is by definition $Hdg(H,\iota)(\frac{htH'}{d})$, and since $HN(H,\iota)$ is the convex hull $(\frac{htH'}{d},\frac{degH'}{d})$ as $(H',\iota)$ varies as a subgroup of $(H,\iota)$ (\cite{F2}, proposition 7), we therefore get
\[Hdg(H,\iota)\geq HN(H,\iota) .\]
\end{proof}

Combined with proposition 3.3, the remark below definition 3.4, proposition 3,10 and remark 3.9, we get

\begin{corollary}Let $(H,\iota,\lambda)$ be a $p$-divisible group with additional structures over $O_K$. Then for
all $x\in [0,htH/d]$,
\[\begin{split}
HN(H,\iota,\lambda)(x)&\leq \frac{1}{m}HN(H[p^m],\iota,\lambda)(mx)\leq HN(H[p],\iota,\lambda)(x)\\
&\leq Hdg(H[p],\iota,\lambda)(x)=Hdg(H_{k},\iota,\lambda)(x).
\end{split} \]
\end{corollary}

\section{Admissible filtered isocrystals}

In this section, let $K|\Q_p$ be a complete field extension for a discrete valuation, with residue field $k$ perfect, and $K_0=\textrm{Frac}W(k)$. We will explain the inequalities in corollary 3.7 for the discrete valuation base case in terms of filtered isocrystals with additional structures.

  First, we review the classical case, i.e. $G=GL_n$, there is no additional structures. Consider the category $\textrm{FilIsoc}_{K|K_0}$ of filtered isocrystals over $K$. The objects are in the form of triplets $(V,\varphi,\textrm{Fil}^\bullet V_K)$, where
\begin{itemize}
\item $(V,\varphi)$ is an isocrystal over $k$;
\item $\textrm{Fil}^\bullet V_K$ is a filtration of $V\otimes_{K_0} K$ such that $\textrm{Fil}^i V_K=0$ for $i>>0$ and $\textrm{Fil}^i V_K=V_K$ for $i<<0$.
\end{itemize}
Recall we have three functions on this category:
\[ht,t_N,t_H: \textrm{FilIsoc}_{K|K_0}\longrightarrow \Z ,\]
where for an object $(V,\varphi, \textrm{Fil}^\bullet V_K) \in \textrm{FilIsoc}_{K|K_0}$,
\[
ht(V,\varphi, \textrm{Fil}^\bullet V_K)= dim_{K_0}V ,\]
\[t_N(V,\varphi, \textrm{Fil}^\bullet V_K)=t_N(V,\varphi)=\sum_{\lambda\in\Q}\lambda dim_{K_0}V_\lambda\]
\[=\textrm{the (vertical coordinate of the) terminal point of the Newton polygon of} (V,\varphi) ,\]
here $V=\bigoplus_{\lambda\in\Q}V_\lambda$ is the Manin-Dieudonn\'e decomposition of $V$ into isocline subspaces $V_\lambda$ of slope $\lambda$,
\[t_H(V,\varphi, \textrm{Fil}^\bullet V_K)=t_H(V_K,\textrm{Fil}^\bullet V_K)=\sum_{i\in\Z}idim_K(gr^i_{\textrm{Fil}^\bullet}V_K)\]
\[=\textrm{the (vertical coordinate of the) terminal point of the Hodge polygon of} (V_K,\textrm{Fil}^\bullet V_K).
\]Consider the functions \[deg_1=t_H-t_N: \textrm{FilIsoc}_{K|K_0}\longrightarrow \Z\] and $\mu_1=\frac{deg_1}{ht}$, then the objects in
$\textrm{FilIsoc}_{K|K_0}$ admit the Harder-Narasimhan filtration relative to the slope function $\mu_1$, see \cite{F3} 9.3. The abelian category of (weakly) admissible filtered isocrystals in the sense of Fontaine is then
\[\textrm{FilIsoc}_{K|K_0}^{ad}=\textrm{FilIsoc}_{K|K_0}^{\mu_1-ss,0} ,\]which is equivalent to the category of crystalline representations of the Galois group $Gal(\overline{K}/K)$.

On this abelian category of (weakly) admissible filtered isocrystals, we have two functions:
\[ht, -t_N(=-t_H): \textrm{FilIsoc}_{K|K_0}^{ad}\longrightarrow \Z .\] Let $\mu=\frac{-t_N}{ht}$, then the objects in $\textrm{FilIsoc}_{K|K_0}^{ad}$ admit the Harder-Narasimhan filtration relative to the slope function $\mu$. In section 9 of loc. cit. Fargues introduced a Harder-Narasimhan filtration of crystalline representations, by considering the larger category of Hodge-Tate representations and its link with categories of filtered vector spaces. For a $p$-divisible group $H$ over $O_K$ one has the equality of polygons $HN(H)=HN(V_p(H))$, where $V_p(H)$ is the rational Tate module. Recall the equivalence functors
\[\textrm{FilIsoc}_{K|K_0}^{ad}
\begin{array}{c}\stackrel{V_{cris}}{\longrightarrow}\\\stackrel{\longleftarrow}{D_{cris}}
\end{array}
\textrm{Rep}_{\Q_p}^{cris}(Gal(\overline{K}/K))
\]
defined by
\[V_{cris}(N,\varphi,Fil^\bullet N_K)=\textrm{Fil}^0(N\otimes_{K_0}B_{cris})^{\varphi=id}\]
and\[D_{cris}(V)=(V\otimes_{\Q_p}B_{cris})^{Gal(\overline{K}/K)}.\]
We have the fact that, the Harder-Narasimhan filtrations in $\textrm{Rep}_{\Q_p}^{cris}(Gal(\overline{K}/K))$ introduced by Fargues and in $\textrm{FilIsoc}_{K|K_0}^{ad}$ defined above coincide, cf. 9.4 of loc. cit..

If we denote the category of isocrytals over $k$ by $\textrm{Isoc}(k)$, and use $deg=-t_N, ht=dim, \mu=\frac{-t_N}{dim}$ to develop the formulism of Harder-Narasimhan filtration, then since the functor of forgetting the filtration
\[\textrm{FilIsoc}_{K|K_0}^{ad}\longrightarrow \textrm{Isoc}(k)\]
\[(V,\varphi, \textrm{Fil}^\bullet V_K)\mapsto (V,\varphi)\]
is exact, and preserves the functions of $ht$ and $-t_N$ on these two categories, we have
the following inequality between concave Harder-Narasimhan polygons:
\[HN(V,\varphi, \textrm{Fil}^\bullet V_K)\leq HN(V,\varphi)=:Newt(V,\varphi) .\]
On the other hand, if we denote the category of filtered vector spaces over $K$ by $\textrm{FilVect}_{K|K}$, which admits the Harder-Narasimhan filtration for the functions
\[ht(V,\textrm{Fil}^\bullet V)=dim_K V,\, deg(V,\textrm{Fil}^\bullet V)=-\sum_{i\in \Z}i.dim_K(gr^i_{\textrm{Fil}^\bullet}V)\]
and $\mu=\frac{deg}{ht}$. We have an exact functor
\[\textrm{FilIsoc}_{K|K_0}^{ad}\longrightarrow \textrm{FilVect}_{K|K}\] defined by composition of the forgetting functor and tensor product with $K$
\[\textrm{FilIsoc}_{K|K_0}^{ad}\ra \textrm{FilVect}_{K|K_0}\ra \textrm{FilVect}_{K|K}.\] This functor preserves the height and degree functions, thus we have the inequality between
concave Harder-Narasimhan polygons
\[HN(V,\varphi, \textrm{Fil}^\bullet V_K)\leq HN(V_K,\textrm{Fil}^\bullet V_K)=:Hdg(V_K,\textrm{Fil}^\bullet V_K) .\]
In fact since $(V,\varphi, \textrm{Fil}^\bullet V_K)$ is admissible, we have the inequality by definition
\[Newt(V,\varphi)\leq Hdg(V_K,\textrm{Fil}^\bullet V_K).\]
To summarize, we get the following inequalities of concave polygons associated to an admissible filtered isocrystal $(V,\varphi, \textrm{Fil}^\bullet V_K)$:
\[HN(V,\varphi, \textrm{Fil}^\bullet V_K)\leq Newt(V,\varphi)\leq Hdg(V_K,\textrm{Fil}^\bullet V_K) .\]

In particular, if $H/O_K$ is a $p$-divisible group over $O_K$, we have the filtered isocrystal $
(N,p^{-1}\varphi, \textrm{Fil}^\bullet N_K)\in \textrm{FilIsoc}_{K|K_0}^{ad,[-1,0]}$ associated to it, and the exact sequence of $K$-vector spaces
 \[0\ra \omega_{H^D,K}\ra N_K\ra Lie(H)_K\ra 0.\]
 Here $\textrm{Fil}^0 N_K=\omega_{H^D,K}=Hom(Lie(H^D)_K,K)$, $\textrm{FilIsoc}_{K|K_0}^{ad,[-1,0]}$ is the full subcategory of $\textrm{FilIsoc}_{K|K_0}^{ad}$, consisting of objects $(V,\varphi, \textrm{Fil}^\bullet V_K)$ of the form that $\textrm{Fil}^{-1}V_K=V_K, \textrm{Fil}^0V_K\subset V_K,\textrm{Fil}^1V_K=0 $. We use the covariant isocrystal $(N,\varphi)$ of $H_k$ here, thus $\varphi=\mathbb{D}(V)\otimes K_0$ for the Verschiebung $V$ of $H_k$. Under the covariant functor $V_{cris}$ of Fontaine
\[V_{cris}: \textrm{FilIsoc}_{K|K_0}^{ad}\ra \textrm{Rep}_{\Q_p}^{cris}(Gal(\overline{K}/K)),\]
we have the equality $V_{cris}(N,p^{-1}\varphi, \textrm{Fil}^\bullet N_K)=V_p(H)$ for the rational Tate module $V_p(H)$.
Then we can identify the following various polygons:
\[\begin{split} &HN(H)=HN(N,p^{-1}\varphi,\textrm{Fil}^\bullet N_K)\\
&Newt(H_k)=Newt(N,p^{-1}\varphi)\\
&Hdg(H_k)=Hdg(N_K,\textrm{Fil}^\bullet N_K) .
\end{split}\]
Here the Hodge polygons in the two sides of the last equality are both the polygon which is the line of slope 1 between the points (0,0) and $(d,d)$, and the line of slope 0 between the points $(d,d)$ and $(h,d)$ (assume $dimH=d,htH=h$). Thus the above inequalities become the following
\[HN(H)\leq Newt(H_k)\leq Hdg(H_k) .\]

Now we consider the cases with additional structures. Let $G$ be a reductive group introduced at the beginning of section 3. Let $k$ be an algebraically closed field of characteristic $p$, $L=\textrm{Frac}W(k)$. Then a pair $(b,\mu)$ is said to be admissible filtered isocrystal with $G$-structures, where $b\in G(L)$ and $\mu: \mathbb{G}_{mK}\ra G_K$ is a co-character defined over a finite extension $K$ of $L$, if for all $(V,\rho)\in \textrm{Rep}_{\Q_p}G$, $(V\otimes L, b\sigma, \textrm{Fil}^\bullet_\mu V_K)$ is an admissible filtered isocrystal in $\textrm{FilIsoc}_{K|L}$. Note to check that $(b,\mu)$ is admissible, it suffices to check for a faithful representation $(V,\rho)$, $(V\otimes L, b\sigma, \textrm{Fil}^\bullet_\mu V_K)$ is admissible. Since the tensor product of two semi-stable admissible filtered isocrystal is semi-stable (see \cite{F3}, section 9), for an admissible pair $(b,\mu)$, we can define its Harder-Narasimhan polygon $\nu_{b,\mu}\in N(G)$, such that for any
$(V,\rho)\in \textrm{Rep}_{\Q_p}G$, $\rho_\ast(\nu_{b,\mu})\in N(GL(V))$ is the Harder-Narasimhan polygon of
$(V\otimes L, b\sigma, \textrm{Fil}^\bullet_\mu V_K)$ we have just defined, as in \cite{DOR},\cite{FR}. Similarly, we can define its Hodge polygon $\nu_{\mu}\in N(G)$ such that for any
$(V,\rho)\in \textrm{Rep}_{\Q_p}G$, $\rho_\ast(\nu_{\mu})\in N(GL(V))$ is the Hodge polygon of $(V_K,\textrm{Fil}^\bullet_{\mu}V_K)$. This was already done in the book \cite{DOR}. On the other hand, Kottwitz has defined the Newton polygon $\nu_b\in N(G)$ such that for any $(V,\rho)\in \textrm{Rep}_{\Q_p}G, \rho_\ast(\nu_b)\in N(GL(V))$ is the Newton polygon of $(V,\varphi)$. Thus we have the following inequalities
\[\nu_{b,\mu}\leq \nu_b\leq \nu_{\mu} \] as elements in the ordered sets $N(G)$. Via the injection $N(G)\hookrightarrow N(\textrm{Res}_{F|\Q_p}GL_n)$, we can view these inequalities as inequalities between polygons over $[0,n]$.

To fix notation, we will work in the PEL cases. The EL case can be treated in the same way, which is simpler. So let $(H,\iota,\lambda)/O_K$ be a $p$-divisible group with additional structures, then its associated filtered isocrystal $(N,p^{-1}\varphi,\textrm{Fil}^\bullet N_K)$ admits induced additional structures, which means that we have an action
$\iota: F\ra \textrm{End}(N,\varphi)$ and a perfect pairing $\lan,\ran: N\times N\ra \Q_p(1)$ such that $\textrm{Fil}^0N_K$ is $F$-invariant and totally isotropic under the induced pairing $\lan,\ran$ on $N_K$. We denote it as $(N,p^{-1}\varphi,\textrm{Fil}^\bullet N_K,\iota,\lan,\ran)$ considering its additional structures. This will then determine an admissible filtered isocrystal with additional structures $(b,\mu)$ for $G$, as in \cite{RR}. The inequalities \[\nu_{b,\mu}\leq \nu_b\leq \nu_\mu\] now translate as
\[\begin{split}
HN(N,p^{-1}\varphi,\textrm{Fil}^\bullet N_K, \iota,\lan,\ran)&\leq Newt(N,p^{-1}\varphi,\iota,\lan,\ran)\\
&\leq Hdg(N_K,\textrm{Fil}^\bullet N_K,\iota,\lan,\ran),
\end{split}
\]
where
\[\begin{split}
HN(N,p^{-1}\varphi,\textrm{Fil}^\bullet N_K, \iota,\lan,\ran)&=HN(N,p^{-1}\varphi,\textrm{Fil}^\bullet N_K, \iota)\\
&=\frac{1}{d}HN(N,p^{-1}\varphi,\textrm{Fil}^\bullet N_K)(d\cdot )
\end{split}
\]
\[
\begin{split}
&Newt(N,p^{-1}\varphi,\iota,\lan,\ran)=\frac{1}{d}Newt(N,p^{-1}\varphi)(d\cdot)\\
&Hdg(N_K,\textrm{Fil}^\bullet N_K,\iota,\lan,\ran)=\frac{1}{d}\sum_{i\in \Z/d\Z}\sigma^i(Hdg(N_{1K},\textrm{Fil}^\bullet N_{1K}))
\end{split}\]
via the injection $N(G)\hookrightarrow N(\textrm{Res}_{F|\Q_p}GL_n)$. Here under the action $\iota$ on $N_K$, we have the decomposition as $K$-vector spaces \[N_K=\bigoplus_{i\in \Z/d\Z}N_{iK}, N_{iK}=\{x\in N_K|a\cdot x=\sigma^i(a)x\},\forall i\in \Z/d\Z,\]
 \[\textrm{Fil}^\bullet N_{iK}=\textrm{Fil}^\bullet N_K\cap N_{iK},\]  Note the action of $\sigma$ on the set of polygons $\{Hdg(N_{iK},\textrm{Fil}^\bullet N_{iK})\}_{i\in \Z/d\Z}$ is \[\sigma(Hdg(N_{iK},\textrm{Fil}^\bullet N_{iK}))=Hdg(N_{i+1K},\textrm{Fil}^\bullet N_{i+1K}).\]

 By definitions we have the following identities between various polygons
 \[
\begin{split}
&HN(H,\iota,\lambda)=HN(N,p^{-1}\varphi,\textrm{Fil}^\bullet N_K, \iota,\lan,\ran)\\
&Newt(H_k,\iota,\lambda)=Newt(N,p^{-1}\varphi,\iota,\lan,\ran)\\
&Hdg(H_k,\iota,\lambda)=Hdg(N_K,\textrm{Fil}^\bullet N_K,\iota,\lan,\ran).
\end{split}\]
Thus the above inequalities explain these ones obtained in corollary 3.7:
\[HN(H,\iota,\lambda)\leq Newt(H_k,\iota,\lambda)\leq Hdg(H_k,\iota,\lambda) .\]

\section{Hodge-Newton filtration for $p$-divisible groups with additional structures}

We restrict ourselves to the PEL cases, the EL case can be treated in the same way, which is simpler. Let $K|\Q_p$ be a complete discrete valuation field with residue field $k$ perfect. Let $(H,\iota,\lambda)$ be a $p$-divisible group with additional structures over $O_K$. We make the following basic assumption:\\
 \\
(HN): \textit{$Newt(H_k,\iota,\lambda)$ and $Hdg(H_k,\iota,\lambda)$ possess a contact point $x$ outside their extremal points which is a break point for the polygon $Newt(H_k,\iota,\lambda)$ }.
 \\

Since these two polygons are symmetric, the symmetric point $\hat{x}$ of $x$ satisfies the same assumption. Denote the coordinate of $x$ by $(x_1,x_2)$. Without loss of generality, we may assume $x_1\leq \frac{htH}{2d}$, and note the equality holds if and only if $x=\hat{x}$.

We now use the explanation of the various polygons in terms of the filtered isocrystal with additional structures $(N,p^{-1}\varphi,\textrm{Fil}^\bullet N_K,\iota,\lan,\ran)$ attached to $(H,\iota,\lambda)$:
\[
\begin{split}
&HN(H,\iota,\lambda)=HN(N,p^{-1}\varphi,\textrm{Fil}^\bullet N_K, \iota,\lan,\ran)\\
&Newt(H_k,\iota,\lambda)=Newt(N,p^{-1}\varphi,\iota,\lan,\ran)\\
&Hdg(H_k,\iota,\lambda)=Hdg(N_K,\textrm{Fil}^\bullet N_K,\iota,\lan,\ran).
\end{split}\]
Then the break points $x$ and $\hat{x}$ correspond to decompositions of isocystal with additional structures
\[(N,p^{-1}\varphi,\iota)=(N_1,p^{-1}\varphi,\iota)\oplus (N_1',p^{-1}\varphi,\iota),\]
\[(N,p^{-1}\varphi,\iota)=(N_2,p^{-1}\varphi,\iota)\oplus (N_2',p^{-1}\varphi,\iota),\]
where the Newton polygon of $(N_1,p^{-1}\varphi,\iota)$ (resp. $(N_2,p^{-1}\varphi,\iota)$) corresponds to the part in the polygon of
$(N,p^{-1}\varphi,\iota,\lan,\ran)$ before $x$ (resp. $\hat{x}$), and the Newton polygon of $(N_1',p^{-1}\varphi,\iota)$ (resp. $(N_2',p^{-1}\varphi,\iota)$) corresponds to the part in the polygon of
$(N,p^{-1}\varphi,\iota,\lan,\ran)$ after $x$ (resp. $\hat{x}$).
We consider the induced filtered isocrystal with additional structures $(N_1,p^{-1}\varphi,\textrm{Fil}^\bullet N_{1K},\iota)$, where $
\textrm{Fil}^\bullet N_{1K}=\textrm{Fil}^\bullet N_K\cap N_{1K}$.

\begin{proposition}
The underlying filtered isocrystal of $(N_1,p^{-1}\varphi,\textrm{Fil}^\bullet N_{1K},\iota)$ is admissible.
\end{proposition}
\begin{proof}
We just need to show the equality
\[t_N(N_1,p^{-1}\varphi)=t_H(N_{1K},\textrm{Fil}^\bullet N_{1K}) .\]
First, the admissibility of $(N,p^{-1}\varphi,\textrm{Fil}^\bullet N_K)$ implies that
\[t_H(N_{1K},\textrm{Fil}^\bullet N_{1K})\leq t_N(N_1,p^{-1}\varphi).\]
So we just need to show the inequality
\[t_H(N_{1K},\textrm{Fil}^\bullet N_{1K})\geq t_N(N_1,p^{-1}\varphi).\]
We denote by \[N=\bigoplus_{i=1}^dN^i,N_1=\bigoplus_{i=1}^dN_{1}^i\] for the decomposition of $N$ and $N_1$ by the action of $F$. Then each subspaces $N_K^i$ and $N_{1K}^i$ admit the induced filtration $\textrm{Fil}^\bullet N^i_K=\textrm{Fil}^\bullet N_K\cap N_K^i,\,\textrm{Fil}^\bullet N^i_{1K}=\textrm{Fil}^\bullet N_{1K}\cap N_{1K}^i$.
By the property of Harder-Narasimhan polygons, we have for all $i=1,\dots,d$
\[deg(N_{1K}^i,\textrm{Fil}^\bullet N_{1K}^i)=dim_KN_{1K}^i-dim_K\textrm{Fil}^0N^i_{1K}\leq Hdg(N_K^i,\textrm{Fil}^\bullet N_K^i)(\frac{dim_KN_{1K}}{d}).\]
Thus
\[\begin{split}deg(N_{1K},\textrm{Fil}^\bullet N_{1K})&=-t_H(N_{1K},\textrm{Fil}^\bullet N_{1K})\\
&=dim_KN_{1K}-dim_K\textrm{Fil}^0N_{1K}\\
&\leq\sum_{i=1}^d(dim_KN^i_{1K}-dim_K\textrm{Fil}^0N^i_{1K})\\
&\leq\sum_{i=1}^dHdg(N_K^i,\textrm{Fil}^\bullet N_K^i)(\frac{dim_KN_{1K}}{d})\\
&=-t_N(N_1,p^{-1}\varphi).\end{split}\]
Here the last equality comes from our assumption (HN) and the definitions of $Newt(N,p^{-1}\varphi,\iota,\lan,\ran)$ and $Hdg(N_K,\textrm{Fil}^\bullet N_K,\iota,\lan,\ran)$.
\end{proof}
\begin{corollary}With the above notation,
the Hodge polygon $Hdg(N_{1K},\textrm{Fil}^\bullet N_{1K},\iota)$ equals to the part before the point $x$ of the Hodge polygon $Hdg(N_K,\textrm{Fil}^\bullet N_K,\iota,\lan,\ran)$.
\end{corollary}
\begin{proof}
Indeed, in the proof of proposition 5.1, we get for all $i=1,\dots,d$
\[dim_KN^i_{1K}-dim_K\textrm{Fil}^0N^i_{1K}=Hdg(N_K^i,\textrm{Fil}^\bullet N_K^i)(\frac{dim_KN_{1K}}{d}).\]
Thus for all $i=1,\dots,d$, $Hdg(N_{1K}^i,\textrm{Fil}^\bullet N_{1K}^i)$ is the part before the point $(\frac{dim_KN_{1K}}{d},dim_KN^i_{1K}-dim_K\textrm{Fil}^0N^i_{1K})$ in the polygon $Hdg(N_K^i,\textrm{Fil}^\bullet N_K^i)$. Then by definition we get the corollary.
\end{proof}

Similarly for the point $\hat{x}$ we have an admissible filtered isocrystal $(N_2,p^{-1}\varphi,\textrm{Fil}^\bullet N_{2K})$.

Since $(N_1,p^{-1}\varphi,\textrm{Fil}^\bullet N_{1K})$ is admissible,
\[\begin{split}&deg(N_1,p^{-1}\varphi,\textrm{Fil}^\bullet N_{1K},\iota)=-\frac{1}{d}t_N(N_1,p^{-1}\varphi)=Newt(N,p^{-1}\varphi,\iota,\lan,\ran)(\frac{dim_KN_{1K}}{d})\\
&\leq HN(N,p^{-1}\varphi,\textrm{Fil}^\bullet N_K,\iota,\lan,\ran)(\frac{dim_KN_{1K}}{d})\leq Newt(N,p^{-1}\varphi,\iota,\lan,\ran)(\frac{dim_KN_{1K}}{d}),
\end{split}\]
thus all the inequalities above are in fact equalities. One has similar equalities for $(N_2,p^{-1}\varphi,\textrm{Fil}^\bullet N_{2K})$. We get the following important corollary.
\begin{corollary}
The Harder-Narasimhan polygon of $(H,\iota,\lambda)$ also passes the points $x$ and $\hat{x}$, which are thus also break points of $HN(H,\iota,\lambda)$. Moreover, for $i=1,2, (N_i,p^{-1}\varphi,\textrm{Fil}^\bullet N_{iK})$ appear in the Harder-Narasimhan filtration of $(N,p^{-1}\varphi,\textrm{Fil}^\bullet N_K)$.
\end{corollary}

\begin{theorem}
Let $K|\Q_p$ be a complete discrete valuation field with residue field $k$ perfect, $(H,\iota,\lambda)$ be a $p$-divisible group with additional structures over $O_K$.
Under the basic assumption (HN), there are unique subgroups \[(H_1,\iota)\subset(H_2,\iota)\subset(H,\iota)\] of $(H,\iota)$
as $p$-divisible groups with additional structures over $O_K$, such that

\begin{enumerate}
\item $\lambda$ induces isomorphisms \[(H_1,\iota)\simeq
((H/H_2)^D,\iota'),\]\[(H_2,\iota)\simeq ((H/H_1)^D,\iota'),\]here $\iota'$ is the induced $O_F$-action on $(H/H_i)^D$ for i=1,2;
 \item the induced filtration of the $p$-divisible group with additional structures $(H_k,\iota)$ over $k$ \[(H_{1k},\iota)\subset(H_{2k},\iota)\subset(H_k,\iota)\] is split;
\item
the Newton(resp. Harder-Narasimhan, resp. Hodge) polygons of $(H_{1},\iota),(H_2/H_1,\iota)$, and $(H/H_2,\iota)$ are the parts of the Newton(resp. Harder-Narasimhan, resp. Hodge) polygon of $(H,\iota,\lambda)$ up to $x$, between $x$ and $\hat{x}$, and from $\hat{x}$ on respectively.
\end{enumerate}
When $x=\hat{x}$, then $H_1$ and $H_2$ coincide.

\end{theorem}

\begin{proof}
We use the formula \[
HN(H,\iota,\lambda)=\inf_{n\geq1}\frac{1}{n}HN(H[p^n],\iota,\lambda)(n\cdot )\] and the following inequalities
\[HN(H,\iota,\lambda)\leq \frac{1}{n}HN(H[p^n],\iota,\lambda)(n\cdot )\leq HN(H[p],\iota,\lambda)\leq Hdg(H_k,\iota,\lambda)\] to deduce that, for $n>>0$ large enough, $\hat{x}$ is a break point of the polygons $
\frac{1}{n}HN(H[p^n],\iota,\lambda)(n\cdot )$. We fix such a $n$. Thus there exists a sub-group \[H_n\subset H[p^n],\] which appears in the Harder-Narasimhan filtration of $H[p^n]$, and admits an induced action of $\iota$. We denote it as $(H_n,\iota)$. Consider the family of finite flat group schemes with additional structures
\[(H_{2^kn},\iota)_{k\geq 0}\subset (H[p^{2^kn}],\iota)_{k\geq 0}.\]
Since
\[\frac{1}{2}HN(H[p^{2^{k+1}n}],\iota,\lambda)(2\cdot)\leq HN(H[p^{2^kn}],\iota,\lambda),\]
we have
\[\mu_{max}(H[p^{2^{k+1}n}]/H_{2^{k+1}n})\leq
\mu_{max}(H[p^{2^{k}n}]/H_{2^{k}n})<\mu_{min}(H_{2^kn}),\]
thus
\[Hom(H_{2^kn},H[p^{2^{k+1}n}]/H_{2^{k+1}n})=0\] by proposition 2.3.
In particular, the composition of
\[H_{2^kn}\hookrightarrow H[p^{2^kn}]\hookrightarrow H[p^{2^{k+1}n}]\ra H[p^{2^{k+1}n}]/H_{2^{k+1}n}\]
is 0, i.e. \[H_{2^kn}\subset H_{2^{k+1}n} .\]
Similarly, since
\[\mu_{min}(H_{2^{k+1}n})\geq \mu_{min}(H_{2^kn})>\mu_{max}(H[p^{2^kn}]/H_{2^kn}),\]
thus \[Hom(H_{2^{k+1}n},H[p^{2^kn}]/H_{2^kn})=0.\]In particular, the composition of
\[H_{2^{k+1}n}\hookrightarrow H[p^{2^{k+1}n}]\stackrel{\times p^{2^kn}}{\longrightarrow}H[p^{2^kn}]\ra
H[p^{2^kn}]/H_{2^kn}\] is 0, i.e. \[p^{2^kn}(H_{2^{k+1}n})\subset H_{2^kn} .\]

Let $C$ be the scheme theoretic closure in $H[p^{2^kn}]$ of
\[ker(H_{2^{k+1}n,K}\stackrel{p^{2^kn}}{\longrightarrow}H_{2^kn,K}),\]
and $D$ be the scheme theoretic closure in $H[p^{2^kn}]$ of
\[im(H_{2^{k+1}n,K}\stackrel{p^{2^kn}}{\longrightarrow}H_{2^kn,K}).\]
Then $H_{2^kn}\subset C$, and we have a sequence
\[0\ra C\ra H_{2^{k+1}n}\ra D\ra 0,\]and
\[deg H_{2^{k+1}n}\leq degC+degD\] with the equality holds if and only if the above sequence is exact, cf. proposition 2.1.

Note \[deg H_{2^{k+1}n}=2degH_{2^kn}\]
\[ht H_{2^{k+1}n}=2htH_{2^kn}.\]
Let $a=htH_{2^kn}\leq ht C$, since $htC+htD=htH_{2^{k+1}n}=2a$, we have $htD\leq a$.

Consider the non-normalized Harder-Narasimhan polygon $HN(H[p^{2^kn}])$ of $H[p^{2^kn}]$, we have
\[degC\leq HN(H[p^{2^kn}])(htC)\]
\[degD\leq HN(H[p^{2^kn}])(htD),\]thus
\[\begin{split}
degC+degD&\leq HN(H[p^{2^kn}])(htC)+HN(H[p^{2^kn}])(htD)\\
&\leq 2HN(H[p^{2^kn}])(\frac{htC+htD}{2})\\
&=2HN(H[p^{2^kn}])(htH_{2^kn})\\
&=2degH_{2^kn}\\
&=degH_{2^{k+1}n}.
\end{split}
\]
Thus we have \[deg H_{2^{k+1}n}=degC+degD\] and
\[0\ra C\ra H_{2^{k+1}n}\ra D\ra 0\] is exact.
We claim that \[C=H_{2^kn}.\]
In fact, if $C \supsetneq H_{2^kn}$, we have $D\subsetneq H_{2^kn}$. Then \[degC+degD<2degH_{2^kn},\] a contradiction!
Thus $C=H_{2^kn}$. Similarly $D=H_{2^kn}$. Therefore we have an exact sequence
\[0\ra H_{2^kn}\ra H_{2^{k+1}n}\stackrel{p^{2^kn}}{\longrightarrow}H_{2^kn}\ra0.\]

Now consider
\[H_2=\varinjlim_{k\geq0}H_{2^kn}\] as a fppf sheaf over $O_K$. It is of $p$-torsion by definition. It is also $p$-divisible, i.e. $H_2\stackrel{p}{\longrightarrow}H_2$ is an epimorphism, since $H_2\stackrel{p^{2^kn}}{\longrightarrow}H_2$ is. As $H_2[p]=H_{2^kn}[p]$, and $H_{2^kn}\stackrel{p^{2^{k}n-1}}{\longrightarrow}H_{2^{k}n}[p]$ is an epimorphism, $H_{2^kn}$ is flat over $O_K$, we can deduce that $
H_{2^{k}n}[p]$ is a finite flat group scheme (see \cite{F4}). Therefore, $H_2$ is a $p$-divisible group over $O_K$. By construction, it naturally admits the induced action of $\iota$, and its filtered isocrystal is exactly $(N_2,p^{-1}\varphi,\textrm{Fil}^\bullet N_{2K},\iota)$.
Thus we get a sub $p$-divisible group with additional structures $(H_2,\iota)$ of $(H,\iota)$. Over $k$, the exact sequence
\[0\ra H_{2k}\ra H_k\ra (H/H_2)_k\ra 0\] splits since $H_{2k}\subset H_k$ is a part of the slope filtration of $H_k$, and the Newton (resp. Harder-Narasimhan, resp. Hodge) polygons of $(H_2,\iota)$ and $(H/H_2,\iota)$ are the parts of the Newton (resp. Harder-Narasimhan, resp. Hodge) polygon of $(H,\iota,\lambda)$ up to $\hat{x}$ and from $\hat{x}$ on respectively.

Similarly, for the point $x$, we can construct a sub $p$-divisible group with additional structures $(H_1,\iota)$ of $(H_2,\iota)$, and the Newton (resp. Harder-Narasimhan, resp. Hodge) polygons of $(H_{1},\iota)$ and $(H_2/H_1,\iota)$ are the parts of the Newton (resp. Harder-Narasimhan, resp. Hodge) polygon of $(H,\iota,\lambda)$ up to $x$ and between $x$ and $\hat{x}$ respectively.
The polarization $\lambda: H\stackrel{\sim}{\ra} H^D$ then induces the isomorphisms \[(H_1,\iota)\simeq
((H/H_2)^D,\iota'),\]\[(H_2,\iota)\simeq ((H/H_1)^D,\iota').\]

\end{proof}

\begin{remark}\begin{enumerate}
\item
For all $n\geq1$, we know that the polygon $\frac{1}{n}HN(H[p^n],\iota,\lambda)(n\cdot)$ passes $\hat{x}$. From the proof of the above theorem, $\hat{x}$ is a break point of this polygon for all $n$ large enough (i.e. there exists $n_0>>0$, for all $n\geq n_0$ $\hat{x}$ is a break point of $\frac{1}{n}HN(H[p^n],\iota,\lambda)(n\cdot)$), and $H_2[p^n]$ is a subgroup in the Harder-Narasimhan filtration of $H[p^n]$. In fact by \cite{F2} lemme 7 and \cite{F3} lemme 3, we get that $\hat{x}$ is a break point of $\frac{1}{n}HN(H[p^n],\iota,\lambda)(n\cdot)$ for all $n\geq 1$, and $H_2[p^n]$ is the subgroup in the Harder-Narasimhan filtration of $H[p^n]$ corresponding to $\hat{x}$ for all $n\geq 1$. Similarly $x$ is also a break point of $\frac{1}{n}HN(H[p^n],\iota,\lambda)(n\cdot)$ for all $n\geq 1$, and $H_1[p^n]$ is the subgroup in the Harder-Narasimhan filtration of $H[p^n]$ corresponding to $x$ for all $n\geq 1$.
\item In the above proof, we just need the fact that the polygon $HN(H,\iota,\lambda)$ also passes the points $x$ and $\hat{x}$, then based on this we can use the theory of Harder-Narasimhan filtration of finite flat group schemes to find $H_1$ and $H_2$. Thus we can prove the theorem over a general complete rank one valuation ring $O_K|\Z_p$, once we can prove that under our assumption (HN) $HN(H,\iota,\lambda)$ also passes the points $x$ and $\hat{x}$, see the following.
\end{enumerate}
\end{remark}

For the application to the cohomology of Rapoport-Zink spaces as in the next section, we will need a stronger version of the above theorem, namely the case $K|\Q_p$ is a complete field extension for a general rank one valuation, not necessarily discrete. For some technical reason we introduce some ``reasonable'' class of $p$-divisible groups over such bases.

\begin{definition}[\cite{F3}, D\'efinition 25]
Let $K|\Q_p$ be a complete field extension for a rank one valuation, $O_K$ be the ring of integers. Suppose the residue field $k$ of $O_K$ is perfect. Let $\alpha: k\hookrightarrow O_K/pO_K$ be the Teichm\"{u}ller section of the projection $O_K/pO_K\ra k$. Let $H$ be a $p$-divisible group over $O_K$, $H_k=H\otimes_{O_K}k$ be its special fiber. We say $H$ is modular, if the identity map $H_k\ra H_k$ lifts to a quasi-isogeny (not necessarily unique)
\[H_k\otimes_{k,\alpha}O_K/pO_K\ra H\otimes_{O_K}O_K/pO_K.\]
When the residue field $k$ is not necessarily perfect, a $p$-divisible group $H$ over $O_K$ is called modular, if for some algebraic closure $\overline{K}$ of $K$, $H\otimes O_{\widehat{\overline{K}}}$ is modular in the above sense.
\end{definition}

We refer the reader to the various lists of equivalent formulations in proposition 22 of loc. cit. In particular, any $p$-divisible group over a complete discrete valuation ring is modular.

\begin{theorem}
Let $K|\Q_p$ be a complete field extension for a rank one valuation, $O_K$ be the ring of integers. Let $(H,\iota,\lambda)$ be a $p$-divisible group with additional (PEL) structures over $O_K$, with the underlying $p$-divisible group $H$ modular. Assume $(H,\iota,\lambda)$ satisfies the assumption (HN), then the same conclusions as theorem 5.4 (1) and (3)  hold for $(H,\iota,\lambda)$, and if the residue filed $k$ of $O_K$ is perfect, then the conclusion (2) of theorem 5.4 also holds for $(H,\iota,\lambda)$.
\end{theorem}
\begin{proof}
In fact, we just need to show under the above conditions, the Harder-Narasimhan polygon $HN(H,\iota,\lambda)$ also passes the contact point $x$ of $Newt(H_k,\iota,\lambda)$ and $Hdg(H_k,\iota,\lambda)$. Then the other arguments in the proof of the above theorem work in the same way in this case, see the remark 5.5 (2).

To show $HN(H,\iota,\lambda)$ also passes the contact point $x$, we use the tool of Rapoport-Zink spaces, see the next section for some review of these spaces.
We may assume $k=\bar{k}$ is algebraical closed. Consider the Rapoport-Zink space of EL type defined as the quasi-isogeny deformation space of $(H_{k},\iota)$. We use $\mathcal{M}$ to denote the Berkovich analytic space over $K_0=FracW(k)$. Then $(H,\iota)$ defines a $K$-valued point of $\M$. For any finite extension $K'|K_0$, the valuation on $K'$ is discrete, and the $p$-divisible groups $(H',\iota)$ over $O_{K'}$ associated to the points of $\M(K')$ satisfy the assumption (HN). Thus the Harder-Narasimhan polygons $HN(H',\iota)$ pass the point $x$. As the rigid points $\M^{rig}=\{y\in \M|[\mathcal{H}(y):K_0]< \infty\}$ are dense in $\M$, and the function of Harder-Narasimhan polygon is semi-continuous over $\M$ (\cite{F3} 13.3), we deduce that the Harder-Narsimhan polygons of the $p$-divisible groups associated any points $y\in \M$ pass $x$. In particular this holds for $HN(H,\iota,\lambda)=HN(H,\iota)$.

\end{proof}

\section{Application to the geometry and cohomology of some non-basic Rapoport-Zink spaces}

The existence and uniqueness of Hodge-Newton filtration can be used to deduce that, the cohomology of simple unramified EL/PEL Rapoport-Zink spaces which satisfy the assumption (HN) contains no supercuspidal representations, as Mantovan did in \cite{M2}, where her assumption was stronger and her results were just stated for the EL case and PEL symplectic case.

Let $(F,V,b,\mu)/(F,\ast,V,\lan,\ran,b,\mu)$ be a simple EL/PEL data with $[F:\Q_p]=d, dim_FV=n$, where $(F,V)/(F,\ast,V,\lan,\ran)$ are as in the section 3 used to define the reductive group $G$. The remaining data $(b,\mu)$ consists of
\begin{itemize}
\item an element $b\in G(L)$ up to $\sigma$-conjugacy (thus we can view $b\in B(G)$, here $B(G)$ is the set of $\sigma$-conjugacy classes in $G(L)$), such that the associated isocrystal $(V\otimes L,b\sigma)$ has slopes in [0,1], thus coming from a $p$-divisible group $\Sigma$ up to isogeny. Here $L=\textrm{Frac}W(\overline{\F}_p)$;
\item a minscule  co-character $\mu: \mathbb{G}_{m\overline{\Q}_p}\ra G_{\overline{\Q}_p}$, up to $G(\overline{\Q}_p)$-conjugacy, such that
\begin{enumerate}
\item $b\in B(G,\mu)$ as an element in $B(G)$, thus the pair $(b,\mu)$ is admissible (\cite{FR}). Here $B(G,\mu)$ is the set defined by Kottwitz (\cite{Ko2});
\item  in the PEL case, $c\circ\mu=id,\,v_p(c(b))=1$, here $v_p$ is the standard valuation on $L$ and $c:G\ra \mathbb{G}_m, c(x)=x^\#x\in R^\times$ for any $\Q_p$-algebra $R$ and $x\in G(R)$ (see section 3).
\end{enumerate}
\end{itemize}

We can make the form of $\mu$ more explicitly. Recall we assume $n=dim_FV$. Fix an $F$-base of $V$ and denote $I_F:=Hom_{\Q_p}(F,\ov{\Q}_p)$. In the EL case, $G=Res_{F|\Q_p}GL_n$ and thus $\mu$ is given by a collection of pairs of integer $(p_\tau,q_\tau)$ such that $p_\tau+q_\tau=n$ for all $\tau\in I_F$:
\[\begin{split}\mu: \mathbb{G}_{m\overline{\Q}_p}&\ra G_{\overline{\Q}_p}\simeq \prod_{\tau\in I_F}GL_{n\ov{\Q}_p}\\z&\mapsto \prod_{\tau\in I_F}diag(\underbrace{z,\cdots,z}_{p_\tau},\underbrace{1,\cdots,1}_{q_\tau}).\end{split}\]In the PEL unitary case, let $\Phi\subset I_F$ be a CM-type, i.e. $\Phi\coprod\Phi\ast=I_F$ where $\Phi\ast=\{\tau\circ\ast|\tau\in\Phi\}$. By definition the group $G$ is such that
\[G_{\ov{\Q}_p}\simeq(\prod_{\tau\in\Phi}GL_{n\ov{\Q}_p})\times\mathbb{G}_{m\ov{\Q}_p}\subset(\prod_{\tau\in I_F}GL_{n\ov{\Q}_p})\times\mathbb{G}_{m\ov{\Q}_p}.\]The $\mu$ is given by a collection of pairs of integer $(p_\tau,q_\tau)_{\tau\in I_F}$ such that $p_{\tau\ast}=q_\tau,q_{\tau\ast}=p_\tau$, and $p_\tau+q_\tau=n$ for all $\tau\in I_F$:
\[\begin{split}\mu: \mathbb{G}_{m\overline{\Q}_p}&\ra G_{\overline{\Q}_p}\subset(\prod_{\tau\in I_F}GL_{n\ov{\Q}_p})\times\mathbb{G}_{m\ov{\Q}_p}\\z&\mapsto \prod_{\tau\in I_F}(\underbrace{z,\cdots,z}_{p_\tau},\underbrace{1,\cdots,1}_{q_\tau})\times (z).\end{split}\]
For the PEL symplectic case, the group $G$ is such that
\[G_{\ov{\Q}_p}=G(\prod_{\tau\in I_F}GSp_n),\]where $G(\prod_{\tau\in I_F}GSp_n)\subset \prod_{\tau\in I_F}GSp_n$ is the subgroup which consists of elements in the product group with the same similitude for all $\tau\in I_F$. In this case the $\mu$ can be given by the following:
\[\begin{split}\mu: \mathbb{G}_{m\overline{\Q}_p}&\ra G_{\overline{\Q}_p}=G(\prod_{\tau\in I_F}GSp_n)\\z&\mapsto \prod_{\tau\in I_F}diag(\underbrace{z,\cdots,z}_{\frac{n}{2}},\underbrace{1,\cdots,1}_{\frac{n}{2}}).\end{split}\]

The element $b\in G(L)$ defines an isocrystal with additional structures $N_b: Rep_{\Q_p}G\ra Isoc(\overline{\F}_p)$ (cf. \cite{DOR},\cite{RR}), in particular for the natural faithful representation $V$ of $G$, $N_b(V)=(V\otimes L,b\sigma)$ is a usual isocrystal, whose Newton polygon after normalization by the action of $F$ is just the image of the element $\nu_b\in N(G)$ defined by Kottwitz under the natural injection $N(G)\hookrightarrow N(Res_{F|\Q_p}GL_FV)$. Here by the normalization of a polygon $\mathcal{P}$ over $[0,dn]$  ($n=dim_FV, d=[F:\Q_p]$), we mean a polygon $\mathcal{P}'$ over $[0,n]$ such that $\mathcal{P}'(x)=\frac{1}{d}\mathcal{P}(dx)$, for all $x\in [0,n]$.
On the other hand, the conjugate class of $\mu$ defines a Hodge polygon (cf. \cite{Ko2},\cite{RR}) \[\nu_{\mu}:=\bar{\mu}=\frac{1}{|\Gamma:\Gamma_\mu|}\sum_{\sigma\in\Gamma/\Gamma_\mu}\sigma(\mu) \in N(G).\] We will  view $\nu_b$ and $\bar{\mu}$ as polygons over $[0,n]$ by the natural injection $N(G)\hookrightarrow N(Res_{F|\Q_p}GL_FV)$. Note the above data defines a $p$-divisible group with additional structures $\Sigma$ over $\overline{\F}_p$ up to isogeny. The polygons $\nu_b$ and $\bar{\mu}$ will be the Newton and Hodge polygons respectively of $\Sigma$. They are also the corresponding polygons of the $p$-divisible groups with additional structures classified by the Rapoport-Zink spaces associated to the above EL/PEL data which we review in the following.

The Rapoport-Zink spaces $\wh{\mathcal{M}}$ associated to the simple unramified EL/PEL data $(F,V,b,$ $\mu)/(F,\ast,V,\lan,\ran,b,\mu)$ are formal schemes locally  formally of finite type over $SpfW(\overline{\F}_p)$, as deformation spaces of $p$-divisible groups with additional structures by quasi-isogenies. More precisely, let $\textrm{Nilp}W(\overline{\F}_p)$ be the category of schemes over $W(\overline{\F}_p)$ over which $p$ is locally nilpotent, then  for any scheme $S\in \textrm{Nilp}W(\overline{\F}_p)$, in the EL case $\wh{\mathcal{M}}(S)=\{(H,\iota,\beta)\}/\sim$; and in the PEL cases $\wh{\mathcal{M}}(S)=\{(H,\iota,\lambda,\beta)\}/\sim$, where
\begin{itemize}
\item $H/S$ is a $p$-divisible group;
\item $\iota: O_F\ra End(H)$ is an action such that
\[det_{O_S}(a,Lie(H))=det(a,V_0), \forall a\in O_F,\]here $V_0$ is the weight 0 subspace of $V_{\overline{\Q}_p}$ defined by $\mu$;
\item $\beta: \Sigma_{\overline{S}}\ra H_{\overline{S}} $ is an $O_F$-equivariant quasi-isogeny, here $\overline{S}\subset S$ is the closed subscheme defined by killing $p$;
\item in the PEL cases, $\lambda: H\ra H^D$ is a polarization, compatible with the action $\iota$, and whose pullback via $\beta$ is the polarization on $\Sigma$ up to a $p$ power scalar multiple. Here as before $H^D$ is the Cartier-Serre dual of $H$.
\item $\sim$ is the relation defined by isomorphisms of $p$-divisible groups with additional structures.
\end{itemize}
Let $J_b(\Q_p)$ be the group of self-quasi-isogenies of $\Sigma$ as $p$-divisible group with additional structures over $\overline{\F}_p$, which is in fact the $\Q_p$-valued points of a reductive group $J_b$ defined over $\Q_p$. Then there is an action of $J_b(\Q_p)$ on $\wh{\mathcal{M}}$ defined by $\gamma\in J_b(\Q_p)$,
\[\gamma:\wh{\M}\ra\wh{\M},\,(\underline{H},\beta)\mapsto (\underline{H},\beta\circ \gamma^{-1}).\]
Let $E$ be the definition field of the conjugate class of $\mu$, the so called reflex field, then there is a non-effective descent datum on $\wh{\mathcal{M}}$ over $O_E$, for details see \cite{RZ}.

Let $\M=\wh{\mathcal{M}}^{an}$ be the Berkovich analytic fiber of $\wh{\M}$ over $L$. Then the local system $\mathcal{T}$ over $\M$ defined by the $p$-adic Tate module of the universal $p$-divisible group on $\wh{\M}$ gives us a tower of Berkovich analytic spaces $(\M_K)_{K\subset G(\Z_p)}$, where for any open compact subgroup $K$ of $G(\Z_p)$, $\M_K$ is the finite \'{e}tale covering of $\M$ parameterizing the $K$-level structures, i.e. the classes modulo $K$ of $O_F$-linear trivialization of $\mathcal{T}$ by $\Lambda$. In particular $\M=\M_{G(\Z_p)}$. The action of $J_b(\Q_p)$ on $\M$ then extends to each rigid analytic space $\M_K$, and the group $G(\Q_p)$ acts on the tower $(\M_K)_{K\subset G(\Z_p)}$ by Hecke correspondences.

We will be interested in the cohomology of the tower $(\M_K)_{K\subset G(\Z_p)}$ of Berkovich analytic spaces. Let $l\neq p$ be a prime number, $\C_p$ be the completion of an algebraic closure of $L$, for any open compact subgroup $K\subset G(\Z_p)$, the $l$-adic cohomology with compact support
\[H^i_c(\M_K\times\C_p,\overline{\Q}_l(D_\M))=\varinjlim_{U}H^i_c(U\times\mathbb{C}_p,\ov{\Q}_l(D_\M))\] was defined, for details see 4.2 of \cite{F1}. Here \[D_\M=dim\M_K=\begin{cases}\sum_{\tau\in I_F}p_\tau q_\tau &\,\textrm{EL case}\\\frac{1}{2}\sum_{\tau\in I_F}p_\tau q_\tau &\,\textrm{PEL unitary case}\\ d\frac{n}{2}(\frac{n}{2}+1)/2 &\,\textrm{PEL symplectic case}.\end{cases}\]
Following Mantovan, we will consider the following groups
\[H^{i,j}(\M_\infty)_\rho:=\varinjlim_{K}\textrm{Ext}^j_{J_b(\Q_p)}(H^i_c(\M_K\times\C_p,\overline{\Q}_l(D_\M)),\rho),\]
for any admissible $\overline{\Q}_l$-representation $\rho$ of $J_b(\Q_p)$. By \cite{M1}, these groups vanish for almost all $i,j\geq 0$, and there is a natural action of $G(\Q_p)\times W_E$ on them. Moreover, as a representation of $G(\Q_p)\times W_E$, $H^{i,j}(\M_\infty)_\rho$ is admissible/continous. For any admissible $\overline{\Q}_l$-representation $\rho$ of $J_b(\Q_p)$, we define a virtual representation of $G(\Q_p)\times W_E$:
\[H(\M_\infty)_\rho=\sum_{i,j\geq 0}(-1)^{i+j}H^{i,j}(\M_\infty)_\rho .\]

To apply our results on the Hodge-Newton filtration, we make as before the following basic assumption :\\
 \\
(HN): \textit{$\nu_b$ and $\bar{\mu}$ possess a contact point $x$ outside their extremal points which is a break point for the polygon $\nu_b$.}
 \\

Thus in the PEL cases, we have a symmetric point $\hat{x}$ of $x$, which satisfies also the above condition. In these cases, if $x=(x_1,x_2)$, we may assume $x_1\leq n/2$.

By the assumption, we can choose decompositions $V=V^1\oplus V^2$ (EL case) or $V=V^1\oplus V^2\oplus V^3$ (PEL cases), such that $N_b(V)=N_b(V^1)\oplus N_b(V^2)$ or $N_b(V)=N_b(V^1)\oplus N_b(V^2)\oplus N_b(V^3)$ is the decomposition of the isocystals corresponding to the break point $x$ or $x$ and $\hat{x}$. In the PEL cases, when $x=\hat{x}$ then $V^2$ is trivial. Let $\Lambda$ be a fixed lattice in $V$ for the EL case and an auto-dual lattice for the PEL cases. Then we can choose decompositions $\Lambda=\Lambda^1\oplus \Lambda^2$ (EL case) or $\Lambda=\Lambda^1\oplus \Lambda^2\oplus \Lambda^3$ (PEL cases), such that they induce the above decompositions for $V$.

Associated to the decompositions $V=\oplus_{i=1}^tV^i, t=2 \,\textrm{or} \,3$, we have a Levi subgroup $M$ of $G$ over $\Q_p$, such that for all $\Q_p$-algebra $R$, \[M(R)=\{g\in G(R)|g \,\textrm{stabilizes} \,V^i_{R}, \forall 1\leq i\leq t\}.\]
Similarly, if we consider the filtrations $0\subset V_1\subset \cdots \subset V_t=V$, $t=2 \,\textrm{or}\, 3$, where $V_i=\oplus_{1\leq j\leq i}V^j, \forall 1\leq i\leq t$, we can define a parabolic subgroup $P$ of $G$ over $\Q_p$, such that for all
$\Q_p$-algebra $R$, \[P(R)=\{g\in G(R)|g\, \textrm{stabilizes}\, V_{iR}, \forall 1\leq i\leq t\}.\] Clearly, $M\subset P$, we denote by $P=MN$ the Levi decomposition of $P$, here $N$ is the unipotent radical of $P$. By definition, we have $b\in M(L)\subset P(L)\subset G(L)$ up to $\sigma$-conjugacy. There is an element $\omega_b$ in the absolute Weyl group of $G$, such that $\omega_b\mu$ factors through $M$ and up to $\sigma$-conjugacy $b\in B(M,\omega_b\mu)$ (the Kottwitz set, cf. \cite{Ko2}), see \cite{M2} or \cite{Ha}. In our special case, we can always choose $\omega_b=1$. The above choices of decompositions of lattices imply $M, P$ are unramified.

Mantovan in \cite{M2} constructed two other type Rapoport-Zink spaces $\wh{\P}$ and $\wh{\Fm}$ for the data $(M,b,\mu)$ and $(P,b,\mu)$ respectively. We briefly recall the definition of these spaces. Both are formal schemes of formally locally of finite type over $SpfW(\overline{\F}_p)$, and classify some type of $p$-divisible groups with additional structures. More precisely, for any $S\in \textrm{Nilp}W(\overline{\F}_p)$, $\wh{\P}(S)=\{(H^i,\iota^i,\beta^i)_{1\leq i\leq t}\}/\sim$ in the EL case, and $\wh{\P}(S)=\{(H^i,\iota^i,\lambda^i, \beta^i)_{1\leq i\leq t}\}/\sim$ in the PEL cases, where
\begin{itemize}
\item $H^i/S$ are $p$-divisible groups;
\item $\iota^i: O_F\ra End(H^i)$ are actions of $O_F$ on $H^i$;
\item $\beta^i: \Sigma^i_{\overline{S}}\ra H^i_{\overline{S}}$ are quasi-isogenies, commuting with the action of $O_F$;
\item in the PEL cases, $\lambda^i: H^i\ra (H^j)^D, i+j=t+1$, are isomorphisms and $(\lambda^i)^D=-\lambda^j$;
such that
\begin{enumerate}
\item $det_{O_S}(a,Lie(H^i))=det(a,V_0^i), \forall a \in O_F, 1\leq i\leq t$;
\item in the PEL cases, there exists $c\in \Q_p^\times$ such that $\lambda^i=c (\beta^{jD})^{-1}\circ \phi^i\circ (\beta^i)^{-1}$ for all $i,j$ such that $i+j=t+1$. Here $\phi^i: \Sigma^i\ra (\Sigma^j)^D$ are the isomorphisms induced by the polarization $\phi: \Sigma\ra\Sigma^D$, for $i=1,\dots,t$.
\end{enumerate}
\item $\sim$ is the relation defined by isomorphisms.
\end{itemize}
As the case of the tower of Rapoport-Zink spaces $(\M_K)_{K\subset G(\Z_p)}$, we may consider the Berkovich analytic fiber $\P=\wh{\P}^{an}$ of $\wh{\P}$, and use the local system provided by the universal Tate module on $\P$ to construct a tower of Berkovich analytic spaces $(\P_K)_{K\subset M(\Z_p)}$ indexing by open compact subgroups $K\subset M(\Z_p)$. These spaces in fact can be decomposed as product of some smaller Rapoport-Zink spaces defined by the EL/PEL data $(F,V^i,b^i,(\omega_b\mu)^i)/(F,\ast,V^i,\lan,\ran,b^i,(\omega_b\mu)^i)$, for more details see section 3 of \cite{M2}. There are natural actions of $J_b(\Q_p)$ on each spaces $\P_K$, and the group $M(\Q_p)$ acts on the tower $(\P_K)_{K\subset M(\Z_p)}$ as Hecke correspondences. Similarly, there is a non-effective descent datum on each of these spaces over $E$.

The filtration $0\subset N_b(V_1)\subset\cdots\subset N_b(V_t)=N_b(V)$ induces a filtration of $p$-divisible groups with additional structures over $\overline{\F}_p$:
\[0\subset \Sigma_1\subset\cdots\subset \Sigma_t=\Sigma.\]
 For any $S\in \textrm{Nilp}W(\overline{\F}_p)$, $\wh{\Fm}(S)=\{(H,\iota, H_\bullet, \beta)\}/\sim$ in the EL case, and $\wh{\Fm}(S)=\{(H,\iota, \lambda, H_\bullet, \beta)\}/\sim$ in the PEL cases, where

\begin{itemize}
\item $H/S$ is a $p$-divisible group;
\item $\iota: O_F\ra End(H)$ is an action of $O_F$ on $H$;
\item in the PEL cases, $\lambda: H\ra H^D$ is an isomorphism compatible with the action of $O_F$;
\item $H_\bullet=(0\subset H_1\subset\cdots\subset H_t=H)$ is an increasing filtration of $H$ by $O_F$-sub-$p$-divisible groups over $S$, such that in the PEL cases $\lambda$ induces isomorphisms $H_i\simeq (H/H_j)^D$ for $i+j=t+1$;
\item $\beta: \Sigma_{\overline{S}}\ra H_{\overline{S}}$ is a quasi-isogeny of $p$-divisible groups with additional structures, and compatible with the filtration, i.e. $\beta(\Sigma_{j\overline{S}})\subset H_{j\overline{S}}$ for any $j=1,\cdots,t$;
satisfying the following conditions
\begin{enumerate}
\item the restrictions of $\beta$ to the $p$-divisible subgroups defining the filtration
\[\beta_j: \Sigma_{j\overline{S}}\ra H_{j\overline{S}}\] are quasi-isogenies;
\item \[det_{O_S}(a,Lie(H_j))=det(a,V_{0j}), \,\forall a \in O_F, j=1,\cdots,t.\]
\end{enumerate}
\item $\sim$ is the relation defined by isomorphisms.
\end{itemize}

As usual, we consider the Berkovich analytic fiber $\Fm=\wh{\Fm}^{an}$ of $\Fm$ over $L$, and we can construct a tower of Berkovich analytic spaces $(\Fm_K)_{K\subset P(\Z_p)}$ indexing by open compact subgroups $K\subset P(\Z_p)$, see \cite{M2} definition 10. There are then natural actions of $J_b(\Q_p)$ on each spaces $\Fm_K$, and the group $P(\Q_p)$ acts on the tower $(\Fm_K)_{K\subset P(\Z_p)}$ as Hecke correspondences. Moreover, there is a non-effective descent datum of $\Fm_K$ over $E$.

As the case of Rapoport-Zink spaces, we will consider the groups
\[H^{i,j}(\P_\infty)_\rho:=\varinjlim_{K}\textrm{Ext}^j_{J_b(\Q_p)}(H^i_c(\P_K\times\C_p,\overline{\Q}_l(D_\P)),\rho)\]
\[H^{i,j}(\Fm_\infty)_\rho:=\varinjlim_{K}\textrm{Ext}^j_{J_b(\Q_p)}(H^i_c(\Fm_K\times\C_p,\overline{\Q}_l(D_\Fm)),\rho)\]
for any admissible $\overline{\Q}_l$-representation $\rho$. Here $D_\P$ (resp. $D_\Fm$) is the dimension of $\P$ (resp. $\Fm$). These groups vanish for almost all $i,j\geq 0$, and as $M(\Q_p)\times W_E$ and $P(\Q_p)\times W_E$ representations respectively are both admissible/continous, cf. \cite{M2} theorem 12. We consider the virtual representations
\[H(\P_\infty)_\rho=\sum_{i,j\geq 0}(-1)^{i+j}H^{i,j}(\P_\infty)_\rho\]
\[H(\Fm_\infty)_\rho=\sum_{i,j\geq 0}(-1)^{i+j}H^{i,j}(\Fm_\infty)_\rho.\]
We would like to compare these representations with
\[H(\M_\infty)_\rho=\sum_{i,j\geq 0}(-1)^{i+j}H^{i,j}(\M_\infty)_\rho .\]
This is achieved by considering the relations between the three towers of Berkovich analytic spaces: $(\M_K)_{K\subset G(\Z_p)},(\Fm_K)_{K\subset P(\Z_p)},(\P_K)_{K\subset M(\Z_p)}$. More precisely, we have the following diagram of morphisms of Berkovich analytic spaces:

\[\xymatrix{
&\Fm\ar[ld]^{\pi_1}\ar[rd]_{\pi_2}&\\
\P\ar@/^1pc/[ru]^{s}& &\M }\]
where in the PEL cases\[\begin{split}&s: (H^i,\iota^i,\lambda^i,\beta^i)_{1\leq i\leq t}\mapsto (\oplus_{i=1}^tH^i,\oplus_{i=1}^t\iota^i,\oplus_{i=1}^t\lambda^i,H_\bullet,\oplus_{i=1}^t\beta^i)\\
&\pi_1: (H,\iota,\lambda,H_\bullet,\beta)\mapsto (gr^iH,\iota^i,\lambda^i,\beta^i)_{1\leq i\leq t}\\
&\pi_2: (H,\iota,\lambda,H_\bullet,\beta)\mapsto (H,\iota,\lambda,\beta), \end{split}\]
here the filtration $H_\bullet$ in the right hand side of the first arrow is the natural one, the $\iota^i,\lambda^i,\beta^i$ in the right hand side of the second arrow are induced by $\iota,\lambda,\beta$ on the graded pieces.

By construction, we have the following facts.
\begin{proposition}[\cite{M2}, Prop. 14, 28, Theorem 36 (3)]
\begin{enumerate}
\item $s$ is a closed immersion;
\item $\pi_1$ is a fibration in balls;
\item $\pi_2$ is a local isomorphism onto its image.
\end{enumerate}
\end{proposition}

In fact, to find the relation between the cohomology groups $H(\P_\infty)_\rho$ and
$H(\Fm_\infty)_\rho$, one has to consider the geometry between $\Fm_K$ and $\P_K:=\P_{K\cap M(\Q_p)}$ for any open compact subgroup $K\subset P(\Z_p)$. We extend the action of $M(\Q_p)$ on the tower $(\P_K)_{K\subset M(\Z_p)}$ to an action of $P(\Q_p)$ on this tower with the unipotent radical of $P(\Q_p)$ acts trivially. In this case, there are $J_b(\Q_p)\times P(\Q_p)$-equivariant closed immersions
\[s_K: \P_K\longrightarrow \Fm_K\]commute with the descent data, for $K\subset P(\Z_p)$ varies. Moreover, there are $J_b(\Q_p)\times P(\Q_p)$-equivariant morphisms of analytic spaces
\[\pi_{1K}: \Fm_K\longrightarrow \P_K\]commute with the descent data, for $K\subset P(\Z_p)$ varies, such that \[\pi_{1K}\circ s_K=id_{\P_K}.\]
 For $K\subsetneq P(\Z_p)$, $\pi_{1K}$ are not necessarily fibrations in balls and their fibers may change. Mantovan's solution of this problem is that for each integer $m\geq 1$, she introduces a formal scheme $j_m: \widehat{\Fm}_m\ra\widehat{\Fm}$ over $\widehat{\Fm}$, such that for any morphism of formal schemes $f: S\ra \widehat{\Fm}$, the $p^m$-torsion subgroup $f^\ast\mathcal{H}[p^m]$ is split if and only if $f$ factors through $j_m$. Here $\mathcal{H}$ is the universal $p$-divisible group over $\widehat{\Fm}$. By definition, one has a formal model $\widehat{\pi}_1:\widehat{\Fm}\ra\widehat{\P}$ of $\pi_1$.
Then one has the fact that formal schemes $\widehat{\Fm}_m$ and $\widehat{\Fm}$ are isomorphic when considered as formal schemes over $\widehat{\P}$, via $\widehat{\pi}_1\circ j_m$ and $\widehat{\pi}_1$ respectively, cf. \cite{M2} proposition 30 (2). Thus the formal schemes $\widehat{\Fm}_m$ can be viewed as some twisted version of $\widehat{\Fm}$. Let $K=K_m:=ker(P(\Z_p)\ra P(\Z_p/p^m\Z_p))$ for the natural projection $P(\Z_p)\ra P(\Z_p/p^m\Z_p)$, and $\Fm_m$ be the analytic generic fiber of $\widehat{\Fm}_m$, then one can define a cover $f_{mK}:\Fm_{mK}\ra \Fm_m$ by the pullback of $\Fm_K\ra \Fm$, i.e.
\[\Fm_{mK}=\Fm_K\times_{\Fm,j_{m\eta}}\Fm_m.\]Let $j_{mK}:\Fm_{mK}\ra\Fm_K$ be the natural projection, we have the following cartesian diagram \[\xymatrix{\Fm_{mK}\ar[r]^{f_{mK}}\ar[d]^{j_{mK}}&\Fm_m\ar[d]^{j_{m\eta}}\\
\Fm_K\ar[r]&\Fm.}\]
 On the other hand one can also define a cover $f_{mK}': \Fm_{mK}'\ra \Fm_m$ by the pullback of $\P_K\ra\P$ via $\pi_1\circ j_{m\eta}: \Fm_m\ra \P$, which is the same with $\pi_1: \Fm_m\ra \P$, i.e.
\[\Fm_{mK}'=\P_K\times_{\P,\pi_1}\Fm_m.\]Let $\pi_{1mK}$ be the natural projection $\Fm_{mK}'\ra \P_K$, we have the following cartesian diagram \[\xymatrix{ \Fm_{mK}'\ar[r]^{f_{mK}'}\ar[d]^{\pi_{1mK}}&\Fm_m\ar[d]^{\pi_{1}}\\
\P_K\ar[r]&\P.}\]
 There are morphisms $\phi_K: \Fm_{mK}\ra\Fm_{mK}'$ and $\varphi_K: \Fm_{mK}'\ra\Fm_{mK}$ such that $\phi_K\circ \varphi_K=id_{\Fm_{mK}'}$ and $\pi_{1mK}\circ\phi_K=\pi_{1K}\circ j_{mK}$. Since $\pi_1$ is a fibration in balls, by the base change theorem for the cohomology with compact support of analytic spaces and proposition 6.1 (2), one has a quasi-isomorphism of cohomological complex
 \[R\Gamma_c(\Fm_{mK}'\times\mathbb{C}_p,\ov{\Q}_l)\simeq R\Gamma_c(\P_K\times\mathbb{C}_p,\ov{\Q}_l(-d))[-2d],\, d=D_{\Fm}-D_{\P}.\]
 In proposition 32 in loc. cit. Mantovan studied the relation between $\Fm_K,\Fm_{mK}$ and $\Fm_{mK}'$, from which she can deduce a quasi-isomorphism
 \[R\Gamma_c(\Fm_{K}\times\mathbb{C}_p,\ov{\Q}_l)\simeq R\Gamma_c(\Fm_{mK}'\times\mathbb{C}_p,\ov{\Q}_l).\]
 Thus one has the following proposition.
 \begin{proposition}[\cite{M2}, Theorem 26]
 For any admissible $\ov{\Q}_l$-representation $\rho$ of $J_b(\Q_p)$, we have an equality of virtural representations of $P(\Q_p)\times W_E$:
 \[H(\P_\infty)_\rho=H(\Fm_\infty)_\rho.\]
 \end{proposition}

To find the relation between $H(\M_\infty)_\rho$ and $H(\P_\infty)_\rho$, we use our main result on the Hodge-Newton filtration. Under our basic assumption, which is weaker than that in \cite{M2}, the results of last section on the existence and uniqueness of Hodge-Newton filtration tell us
\begin{proposition}
$\pi_2$ is bijective, thus it is an isomorphism of Berkovich analytic spaces.
\end{proposition}
\begin{proof}
This is a direct consequence of theorem 5.4, 5.7, and proposition 6.1 (3).
\end{proof}

For an open compact subgroup $K\subset G(\Z_p)$, denote $\Fm_K:=\Fm_{K\cap P(\Q_p)}$, then we have a natural morphism $\pi_{2K}: \Fm_K\ra\M_K$ such that $\pi_{2G(\Z_p)}=\pi_2$ which is defined above. Let $\P_K:=\P_{K\cap M(\Q_p)}$, we have also natural generalizations $s_K, \pi_{1K}$ of $s$ and $\pi$ respectively. Moreover we have the following diagram in level $K$
\[\xymatrix{
&\Fm_K\ar[ld]^{\pi_{1K}}\ar[rd]_{\pi_{2K}}&\\
\P_K\ar@/^1pc/[ru]^{s_K}& &\M_K. }\]
\begin{corollary}With the above notation, $\pi_{2K}$ is a closed immersion, and
we have isomorphisms
\[\M_K\simeq \M_K\times_{\M}\Fm\simeq \coprod_{K\setminus G(\Q_p)/P(\Q_p)}\Fm_{K\cap P(\Q_p)}.\]
\end{corollary}
\begin{proof}
One argues exactly as \cite{M2} 8.2.
\end{proof}

Before passing to cohomological conclusion, we consider some application to the $p$-adic period morphism and monodromy representations. Recall that in \cite{RZ} chapter 5 Rapoport-Zink had defined a $p$-adic period morphism
\[\pi: \M\longrightarrow \Fm^a:=\Fm(G,\mu)^a\subset \Fm(G,\mu)^{an},\]which is $G(\Q_p)$-invariant and $J_b(\Q_p)$-equivariant. Here $\Fm^a$ is the image of $\pi$, which is an open subspace of $\Fm(G,\mu)^{an}$, the associated Berkovich analytic space of $\Fm(G,\mu):=G_L/P_{\mu L}$. Here $P_\mu$ is the parabolic subgroup defined by $\mu$ over the reflex field $E$, and $P_{\mu L}$ is its base change over $L$. The definition of $\pi$ for rigid points is as follow. Associated to a rigid point $x\in \M(K)$ ($K|L$ is thus a finite extension) there is the $p$-divisible group with additional structures $(H,\iota,\lambda)$ over $O_K$ and the quasi-isogeny $\rho: \Sigma_{O_K/pO_K}\ra H_{O_K/pO_K}$, which defines an isomorphism
\[\rho_\ast: (V_L,b\sigma)\stackrel{\sim}{\longrightarrow} (\mathbb{D}(H_k)_L,\varphi).\]Let
\[\textrm{Fil}_{\pi(x)}V_K=\rho_\ast^{-1}(\omega_{H^D,K})\subset V_K\]for the Hodge filtration sequence
\[0\ra \omega_{H^D,K}\ra \mathbb{D}(H_k)_K\ra \textrm{Lie}(H)_K\ra 0.\]
Then by definition $\pi(x)=\textrm{Fil}_{\pi(x)}V_K\in\Fm^a(K)$.

In our situation, we have also the $p$-adic period morphism which is still denoted by $\pi$ by abuse of notation \[\pi: \P\longrightarrow \Fm(M,\mu)^a\subset \Fm(M,\mu)^{an},\] for the Rapoport-Zink $\P$. Let
\[\pi: \Fm\longrightarrow \Fm(G,\mu)^{an}\]be the composition of $\pi_2:\Fm\ra\M$ and $\pi:\M\ra\Fm(G,\mu)^{an}$, and $\Fm^a(P,\mu)$ be its image. We have the following enlarged diagram:
\[\xymatrix{
&\Fm\ar[ld]_{\pi_1}\ar[rd]^{\pi_2}\ar@{->>}[dd]&\\
\P\ar@{->>}[dd]& &\M\ar@{->>}[dd]\\
&\Fm^a(P,\mu)\ar[ld]_{\pi_1'}\ar[rd]^{\pi_2'}&\\
\Fm^a(M,\mu)& &\Fm^a(G,\mu) .
 }\]
Here $P_\mu$ is the parabolic subgroup of $G_{E}$ defining $\Fm(G,\mu)$. We have
\begin{enumerate}
\item $\pi_1'$ is a fibration in affine analytic spaces;
\item $\pi_2'$ is the identity.
\end{enumerate}

Let $\mathcal{T}$ be the $\Z_p$-local system over $\M$ defined by the Tate module of the universal $p$-divisible group over $\widehat{\M}$, then it descends to a $\Q_p$-local system over the $p$-adic period domain $\Fm^a$. Let $\ov{x}$ be a geometric point of $\M$ and $\ov{y}$ be its image under the $p$-adic period morphism $\pi:\M\ra \Fm^a$. Then by \cite{dJ} theorem 4.2, these local systems define monodromy representations
\[\rho_{\ov{x}}: \pi_1(\M,\ov{x})\longrightarrow G(\Z_p)\]and \[\rho_{\ov{y}}:\pi_1(\Fm^a,\ov{y})\longrightarrow G(\Q_p)\] respectively. Here $\pi_1(X,\ov{x})$ is the fundamental group defined by de Jong in loc. cit. for a Berkovich analytic space $X$ and a geometric point $\ov{x}$ of $X$. Then under our basic assumption $(HN)$ and the notations above, the existence of Hodge-Newton filtration implies the following.

\begin{corollary}
The monodromy representations $\rho_{\ov{x}}$ and $\rho_{\ov{y}}$ factor through $P(\Z_p)$ and $P(\Q_p)$ respectively.
\end{corollary}

In \cite{Chen}, M. Chen has constructed some determinant morphisms for the towers of simple unramified Rapoport-Zink spaces. Under the condition that there is no non-trivial contact point of the Newton and Hodge polygons, and assume the conjecture
that \[\pi_0(\widehat{\M})\simeq Im\varkappa\] for the morphism $\varkappa:\widehat{\M}\ra\triangle$ constructed in \cite{RZ} 3.52, Chen proved that the associated monodromy representation under this condition is maximal, and thus the geometric fibers of her determinant morphisms are exactly the geometric connected components, see th\'eor\`eme 5.1.2.1 and 5.1.3.1 of loc. cit.. Our result confirms that the condition that ``there is no non-trivial contact point of the Newton and Hodge polygons'' is thus necessary, see the remark in 5.1.5 of loc. cit.. In the split cases considered in \cite{M3} and \cite{M2}, their results already confirmed that the above condition is necessary.

Now we look at the cohomological consequence of the existence of Hodge-Newton filtration. Proposition 6.2 and corollary 6.4 together imply that
\[\begin{split}
H(\M_\infty)_\rho&=\mathrm{Ind}_{P(\Q_p)}^{G(\Q_p)}H(\Fm_\infty)_\rho \\
&=\mathrm{Ind}_{P(\Q_p)}^{G(\Q_p)}H(\P_\infty)_\rho.
\end{split}
\]
We summarize as the following theorem.
\begin{theorem}
Under the assumption (HN), we have an equality of virtual representations of $G(\Q_p)\times W_E$:
\[H(\M_\infty)_\rho=\mathrm{Ind}_{P(\Q_p)}^{G(\Q_p)}H(\P_\infty)_\rho .\]
In particular, there is no supercuspidal representations of $G(\Q_p)$ appear in the virtual representation
$H(\M_\infty)_\rho$.
\end{theorem}

\section{Application to the cohomology of some Shimura varieties}

The above theorem generalizes the main result of \cite{M2}. As there, we can consider further the application to the cohomology of Newton strata of some more general PEL-type Shimura varieties. We will not pursue the full generalities here, but concentrate on the Shimura varieties studied in \cite{BW} and \cite{VW}, since studying these varieties was one of the motivation of this paper as said in the introduction. They are closely related to Harris-Taylor's Shimura varieties, but the local reductive groups involved are unitary groups.

More precisely, let $Sh_{G(\Z_p)\times K^p}/O_{E_\nu}$ be a smooth PEL-type Shimura variety over the integer ring of $E_\nu$ as in \cite{BW} or \cite{VW}, the local reflex field which is a quadratic unramified extension of $\Q_p$ if $n\neq 2$ and $\Q_p$ if $n=2$. Let $\overline{Sh}_{G(\Z_p)\times K^p}$ be its special fiber, then we have the Newton polygon stratification
\[\overline{Sh}_{G(\Z_p)\times K^p}=\coprod_{b\in B(G,\mu)}\overline{Sh}_{G(\Z_p)\times K^p}^{(b)}.\]Here in this special case, $G_{\Q_p}$ is isomorphic to a simple PEL unitary group in our notions by the Morita equivalence, and $B(G,\mu)$ is in bijection with the set of polygons defined in \cite{VW} (3.1), in particular, any non basic element $b$ satisfy our assumption (HN).

Let $R^j\Psi_\eta(\overline{\Q}_l), j\geq 0$ denote the $l$-adic nearby cycles of some fixed integral models of the Shimura varieties $Sh_{K_p\times K^p}$ with some level structures $K_p$ at $p$, defined for example as in \cite{M1} for Drinfeld level structures, or the book \cite{RZ} for parahoric level structures. Then we have also the Newton polygon stratification for the special fibers with level structures at $p$, and for each $b\in B(G,\mu)$, we have the virtual representation of $G(\mathbb{A}_f)\times W_{E_\nu}$
\[H_c(\overline{Sh}_{\infty}^{(b)}\times\overline{\F}_p,R\Psi_\eta(\overline{\Q}_l)):=\sum_{i,j\geq 0}(-1)^{i+j}\varinjlim_{K_p\times K^p}H^i_c(\overline{Sh}_{K_p\times K^p}^{(b)}\times\overline{\F}_p,R^j\Psi_\eta(\overline{\Q}_l)) .\]

Let $\widehat{Sh}_{G(\Z_p)\times K^p}$ be the $p$-adic completion of $Sh_{G(\Z_p)\times K^p}$, $\widehat{Sh}_{G(\Z_p)\times K^p}^{an}$ be the Berkovich analytic fiber of this formal scheme.
For any open compact subgroup $K_p\subset G(\Z_p)$, let $\widehat{Sh}_{K_p\times K^p}^{an}$ be the \'etale covering of $\widehat{Sh}_{G(\Z_p)\times K^p}^{an}$ defined by trivializing the Tate module in the usual way.  When $K_p$ is a Drinfeld level structure subgroup or a parahoric subgroup, $\widehat{Sh}_{K_p\times K^p}^{an}$ has a formal model: the $p$-adic completion $\widehat{Sh}_{K_p\times K^p}$ of $Sh_{K_p\times K^p}$. Then the theory of formal vanishing cycles tells us we have the equality of cohomology
\[R\Gamma(\overline{Sh}_{K_p\times K^p}\times\overline{\mathbb{F}}_p, R\Psi_\eta(\overline{\Q}_l))=
R\Gamma(\widehat{Sh}_{K_p\times K^p}^{an}\times \C_p,\overline{\Q}_l).\]
Thus we have the equality of virtual representations \[\sum_{i,j\geq 0}(-1)^{i+j}\varinjlim_{K_p\times K^p}H^i(\overline{Sh}_{K_p\times K^p}\times\overline{\mathbb{F}}_p, R^j\Psi_\eta(\overline{\Q}_l))=\sum_{i\geq 0}(-1)^i\varinjlim_{K_p\times K^p}H^i(\widehat{Sh}_{K_p\times K^p}^{an}\times \C_p,\overline{\Q}_l).\]

If $Sh_{G(\Z_p)\times K^p}$ is proper, (for example if $B=V$ as the notations in \cite{BW} and \cite{VW}, these Shimura varieties have the same generic fibers as that of the Shimura varieties of some special cases studied by \cite{HT},) then we have \[\widehat{Sh}_{K_p\times K^p}^{an}=Sh_{K_p\times K^p}^{an},\]where the later is the associated Berkovich analytic space of the Shimura varieties $Sh_{K_p\times K^p}$ over $E_\nu$. Thus we have the
equality of virtual representations
\[\begin{split}\sum_{i\geq0}(-1)^iH^i(Sh_\infty\times\overline{E}_\nu,\overline{\Q}_l):&=\sum_{i\geq 0}(-1)^i\varinjlim_{K}H^i(Sh_K\times \overline{E}_\nu,\overline{\Q}_l)\\
&=\sum_{i\geq 0}(-1)^i\varinjlim_{K}H^i(Sh^{an}_K\times\C_p,\overline{\Q}_l)\\
&=\sum_{i,j\geq 0}(-1)^{i+j}\varinjlim_{K_p\times K^p}H^i(\overline{Sh}_{K_p\times K^p}\times\overline{\mathbb{F}}_p, R^j\Psi_\eta(\overline{\Q}_l))\\
&=\sum_{b\in B(G,\mu)}H_c(\overline{Sh}_{\infty}^{(b)}\times\overline{\F}_p,R\Psi_\eta(\overline{\Q}_l)).
\end{split}\]

The main results of \cite{M1} tell us the cohomology of each Newton polygon strata can be computed in terms of the $l$-adic cohomology of the corresponding Rapoport-Zink spaces and Igusa varieties. More precisely, we have the formula \[\begin{split}&H_c(\overline{Sh}_{\infty}^{(b)}\times\overline{\F}_p,R\Psi_\eta(\overline{\Q}_l))\\
=&\sum_{i,j,k\geq 0}(-1)^{i+j+k}\varinjlim_{K_p}
\mathrm{Ext}^i_{J_b(\Q_p)}(H^j_c(\M_{K_p}\times\C_p,\overline{\Q}_l(D_\M)),H^k_c(Ig_b,\overline{\Q}_l)),
\end{split}\]
see \cite{M1} for the precise definition of the Igusa varieties and their cohomology.
Thus the main results of this paper imply in particular
\begin{corollary}
 For the Shimura varieties studied by \cite{BW},\cite{VW}, for any non-basic strata, the cohomology group $H_c(\overline{Sh}_{\infty}^{(b)}\times\overline{\F}_p,R\Psi_\eta(\overline{\Q}_l))$ can be written as some suitable parabolic induction of virtual representations, and thus contains no supercuspidal representations of $G(\Q_p)$.
\end{corollary}

Finally note that in a recent preprint \cite{I}, Imai and Mieda have proven that, for non-proper Shimura varieties the supercuspidal parts (see \cite{F1} d\'efinition 7.1.4, 8.1.2 for example) of the compactly supported (or intersection) cohomology and nearby cycle cohomology are the same:
\[H_c^i(Sh_\infty\times\overline{E}_\nu,\overline{\Q}_l)_{cusp}=H_c^i(\ov{Sh}_\infty\times \overline{\F}_p,R\Psi_\eta(\overline{\Q}_l))_{cusp}.\]
Since one has the equality of virtual representations \[\sum_{i\geq 0}(-1)^{i}H^i_c(\overline{Sh}_\infty\times\overline{\mathbb{F}}_p, R\Psi_\eta(\overline{\Q}_l))=\sum_{b\in B(G,\mu)}H_c(\overline{Sh}_{\infty}^{(b)}\times\overline{\F}_p,R\Psi_\eta(\overline{\Q}_l)),\]
combined with Mantovan's formula above we find that, once a non-basic Newton polygon has a nontrivial contact point to the Hodge polygon, which is a break point for the Newton polygon, then there is no contribution of the cohomology of this non-basic strata to the supercuspidal part of the cohomology (of whichever kind) of the non-proper Shimura varieties. For example, this is the case for the two non-basic stratas for the Shimura varieties associated to $GSp4$.

\end{document}